\documentclass[11pt]{amsart}
\usepackage{amsthm,amsmath,amssymb}
 
\usepackage{geometry}
\usepackage{graphicx,enumitem}
\usepackage{color}

\numberwithin{equation}{section}
\numberwithin{figure}{section}
\theoremstyle{plain}
\newtheorem{thm}{Theorem}[section]
\newtheorem{lem}[thm]{Lemma}
\newtheorem{cor}[thm]{Corollary}

\theoremstyle{remark}
\newtheorem{rmk}[thm]{Remark}

\newcommand{\M}{\operatorname{T}}
\newcommand{\MM}{\operatorname{M}}

\newcommand{\wt}{\operatorname{wt}}

\newcommand{\PP}{\operatorname{PP}}

\begin{document}

\title[Ratio of tiling generating functions]{Ratio of tiling generating functions of semi-hexagons and quartered hexagons with dents}

\author{Tri Lai}
\address{Department of Mathematics, University of Nebraska -- Lincoln, Lincoln, NE 68588, U.S.A.}
\email{tlai3@unl.edu}
\thanks{This research was supported in part  by Simons Foundation Collaboration Grant (\# 585923).}



\date{\today}

\dedicatory{}

\begin{abstract}
We consider the tiling generating functions of semi-hexagons and quartered hexagons with dents on their sides. In general, there are no simple product formulas for these generating functions. However, we show that the modification in the regions' width changes the tiling generating functions by only a simple multiplicative factor.  
\end{abstract}

\maketitle

\section{Introduction}\label{Sec:Intro}
 In general, even a small modification of a region would lead to an unpredictable change in its tiling number. However, in some situations, it changes the tiling number by only a simple multiplicative factor. The author and Rohatgi first observed this phenomenon for the doubly--dented hexagons and named it the ``Shuffling Phenomenon" for tilings \cite{shuffling}.

The first example of this phenomenon was recognized earlier, in 2018 when the author attended the JMM, San Diego. After discussing with Dennis Stanton about the tiling number of the `\emph{$S$-cored hexagon}' (a hexagon with a cluster of four triangles removed), the author found a striking pattern in the tiling number of the region when the side-lengths of the $S$-core are changed. This example was later generalized in \cite{HoleDent, HoleDent2, CLR}.
The phenomenon has been found in many different forms and different region families. We refer the reader to, e.g. \cite{shuffling2,shuffling3,Byun,Ful,Con,Ful2,Byun2}, for recent work about the phenomenon.

In this paper, we show several new instances of the shuffling phenomenon. In particular, we are investigating two new region families. The first family is a class of \emph{semi-hexagons}, i.e. upper halves of symmetric hexagons (see Figure \ref{Fig:Semitwodent}), and the second family consists of certain \emph{quartered hexagons}, i.e. halves of a symmetric semi-hexagon (see Figure \ref{Fig:Semitwodent5}). We show that the tiling generating functions of these regions change by only a simple multiplicative factor if we adjust the width while fixing the other parameters. A highlight of the result is that the tiling generating functions of these regions are \emph{not} given by simple product formulas themselves. 

We want to emphasize that most of the results in the field of enumeration of tilings are unweighted enumerations, i.e., `plain' counting. The weighted enumerations are \emph{very rare}. This paper is devoted to such rare enumerations. The unweighted version of Theorem 1.1 is independent found by Condon \cite{Con}. Strictly speaking, Condon investigates a different family of regions, hexagons with dents on two non-adjacent sides. However, his regions have the same tiling number as our regions. 

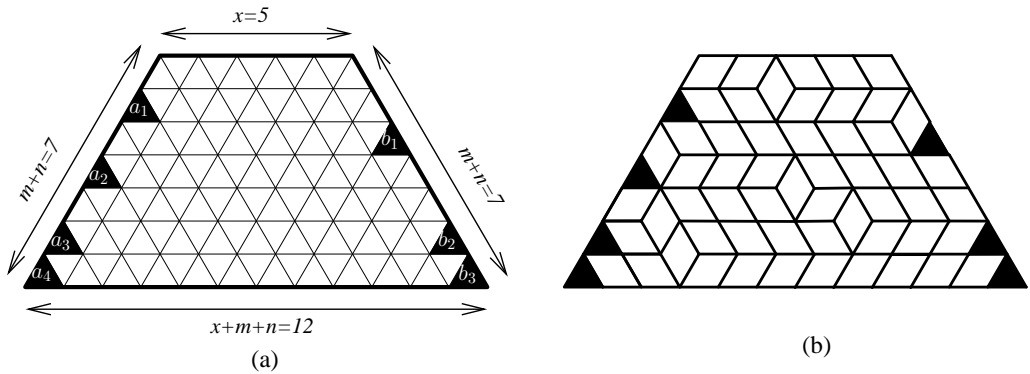
\begin{figure}\centering
\setlength{\unitlength}{3947sp}%
\begingroup\makeatletter\ifx\SetFigFont\undefined%
\gdef\SetFigFont#1#2#3#4#5{%
  \reset@font\fontsize{#1}{#2pt}%
  \fontfamily{#3}\fontseries{#4}\fontshape{#5}%
  \selectfont}%
\fi\endgroup%
\resizebox{!}{5cm}{
\begin{picture}(0,0)%
\includegraphics{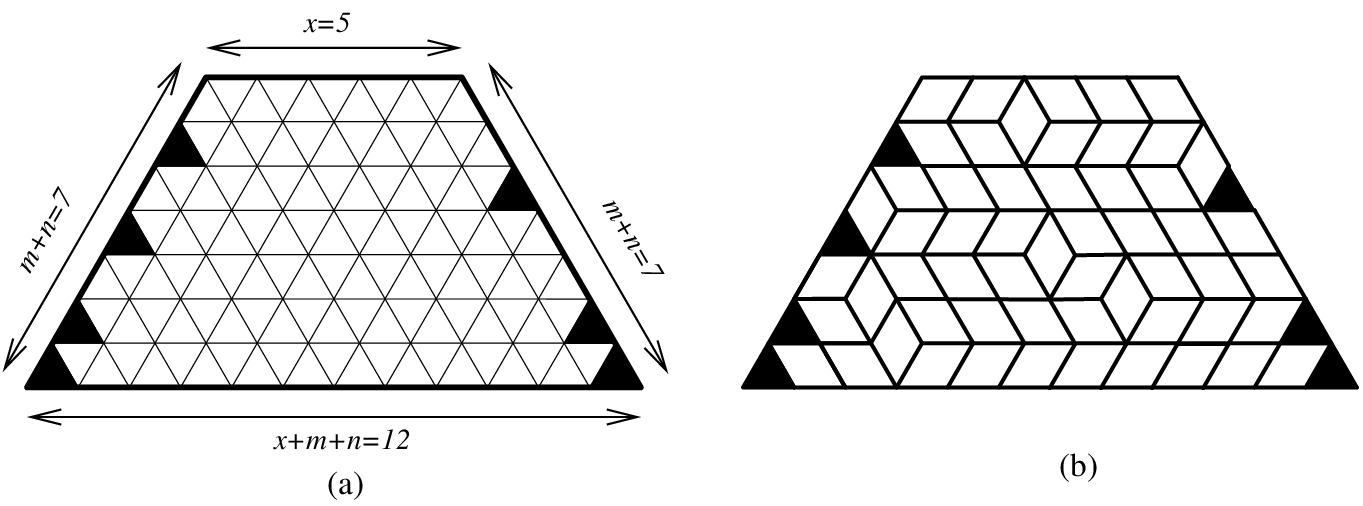}%
\end{picture}%

\begin{picture}(10888,4040)(602,-3652)
\put(1089,-2266){\makebox(0,0)[lb]{\smash{{\SetFigFont{14}{16.8}{\rmdefault}{\mddefault}{\itdefault}{\color[rgb]{1,1,1}$a_3$}%
}}}}
\put(886,-2604){\makebox(0,0)[lb]{\smash{{\SetFigFont{14}{16.8}{\rmdefault}{\mddefault}{\itdefault}{\color[rgb]{1,1,1}$a_4$}%
}}}}
\put(4599,-1186){\makebox(0,0)[lb]{\smash{{\SetFigFont{14}{16.8}{\rmdefault}{\mddefault}{\itdefault}{\color[rgb]{1,1,1}$b_1$}%
}}}}
\put(5206,-2259){\makebox(0,0)[lb]{\smash{{\SetFigFont{14}{16.8}{\rmdefault}{\mddefault}{\itdefault}{\color[rgb]{1,1,1}$b_2$}%
}}}}
\put(5424,-2626){\makebox(0,0)[lb]{\smash{{\SetFigFont{14}{16.8}{\rmdefault}{\mddefault}{\itdefault}{\color[rgb]{1,1,1}$b_3$}%
}}}}
\put(1921,-834){\makebox(0,0)[lb]{\smash{{\SetFigFont{14}{16.8}{\rmdefault}{\mddefault}{\itdefault}{\color[rgb]{1,1,1}$a_1$}%
}}}}
\put(1494,-1561){\makebox(0,0)[lb]{\smash{{\SetFigFont{14}{16.8}{\rmdefault}{\mddefault}{\itdefault}{\color[rgb]{1,1,1}$a_2$}%
}}}}
\end{picture}}
\caption{(a) A semi-hexagon with dents on two sides, and (b) a tiling of its. The black triangles indicate the unit triangles removed.}\label{Fig:Semitwodent}
\end{figure}

\begin{figure}\centering
\includegraphics[width=8cm]{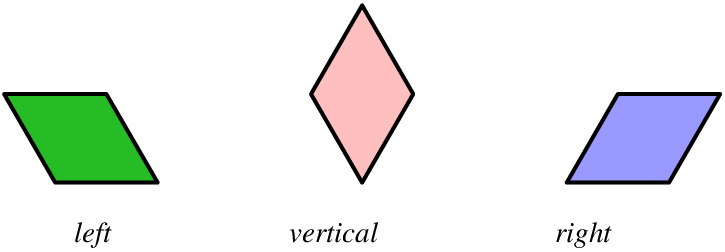}
\caption{Three possible orientations of the lozenges: \emph{left}, \emph{vertical}, and \emph{right} lozenges.}\label{Fig:orientation}
\end{figure}

\begin{figure}\centering
\includegraphics[width=15cm]{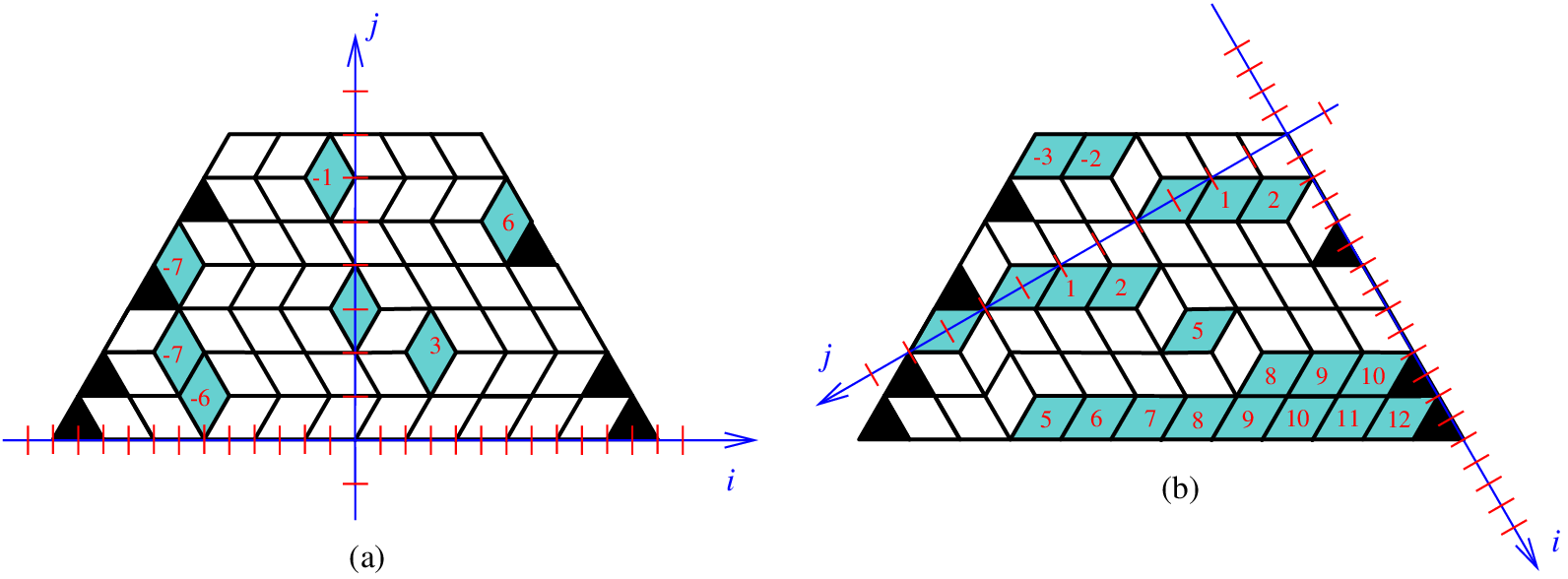}
\caption{Two weight assignments to the lozenges of a semi-hexagon with dents on two sides. The shaded lozenges passed by the $j$-axis are weighted by $\frac{X+Y}{2}$, the ones with label $k$ are weighted by  $\frac{Xq^{k}+Yq^{-k}}{2}$.}\label{Fig:weight}
\end{figure}

We now define in detail the semi-hexagon with dents on two sides. Consider a trapezoidal region of side-lengths $x,m+n,x+m+n,m+n$ (in counter-clockwise order, starting from the top\footnote{From now on, we always list the side-lengths of a region in this order.}) on the triangular lattice, as shown in \ref{Fig:Semitwodent}(a). Next, we remove $m$ up-pointing unit triangles on the left and $n$ unit triangles on the right sides. These removed triangles are called the `\emph{dents}'  and indicated by the black triangles in the figure. A \emph{lozenge} is an union of any two unit triangles that share an edge (see Figure \ref{Fig:orientation} for three possible orientations of the lozenges: \emph{vertical, left and right}). A \emph{lozenge tiling} of a region is a covering of the region by lozenges with no gaps or overlaps. See Figure \ref{Fig:Semitwodent}(b) for a tiling of a semi-hexagon with dents. In general, our dented semi-hexagon may not have any tiling (see Lemma \ref{tileability1} for the tile-ability of this region). Even when it has tilings, the number of tilings is not given by a nice product formula. 

We now consider a weight assignment for the vertical lozenges as follows.  We define a rectangular coordinate system with the horizontal $i$-axis running along the base of the semi-hexagon; the vertical $j$-axis is passing the middle point of the base. A unit on $i$-axis is equal to $1/2$ times the side-length of a lozenge, and a unit on $j$-axis is equal to $\sqrt{3}/2$ times the side-length of a lozenge. The vertical lozenge with center at the point $(i,j)$ is weighted by
 $\frac{Xq^{i}+Yq^{-i}}{2}$, where $X,Y,q$ are three indeterminates. The lozenges of different orientations (left, and right) are all weighted by $1$. We note that our lozenge-weights do \emph{not} depend on the $j$-ordinate. See Figure \ref{Fig:weight}(a) for an example. The \emph{weight} of a tilling is now the product of its lozenge-weights\footnote{From now on, we only define the weight assignment for the lozenges, and the weights of the tilings are obtained implicitly in this way.}.  This weight assignment is a special case of the \emph{elliptic weight} considered in \cite{Borodin}. Denote by $S_{x}(\textbf{a}, \textbf{b})=S_{x}((a_i)_{i=1}^{m}; (b_j)_{j=1}^{n})$ the resulting weighted region, where $\textbf{a}=(a_i)_{i=1}^{m}$ and $\textbf{b}=(b_j)_{j=1}^{n}$ are respectively the sequences of the left dents' positions and the right dents' positions (as they appear from top to bottom). 
 
 All regions considered in this paper are weighted regions. Strictly speaking, the a `\emph{weighted region}' is a pair $(R,\wt)$, where $R$ is an unweighted region on the triangular lattice, called the ``shape" of the region, and $\wt$ is a weight assignment for the tilings of $R$. We will see in the next part of the paper that there exist different weighted regions that have the same shape. Whenever the weight assignment is clearly given, we abuse the notation by viewing $R$ as the weighted region. In the rest of the paper, we use the notation $\M(R)$ for the weighted sum of all tilings of $R$. If $R$ does not have any tiling, then $\M(R)=0$. When $R$ is a degenerated region (i.e., a region with empty interior),
we set $\M(R)=1$ by convention. We call $\M(R)$ the \emph{tiling generating function} of $R$.

In general, the tiling generating function of the semi-hexagon $S_{x}(\textbf{a}, \textbf{b})$ is not given by a simple product formula. However, if we consider the ratio of the tiling generating functions of $S_{x}(\textbf{a}, \textbf{b})$ and its `sibling' $S_{y}(\textbf{a}, \textbf{b})$, then a magical cancelation happens. (Intuitively, $S_{y}(\textbf{a}, \textbf{b})$ is obtained by horizontal stretching or compressing $S_{x}(\textbf{a}, \textbf{b})$.) The ratio reduces to a nice product formula.

We often use the standard  $q$-Pochhammer symbol in our tiling formulas:
\begin{equation}
(x; q)_n:=
\begin{cases}1 &\text{if $n=0$;}\\
(1-x)(1-xq)\cdots(1-xq^{n-1}) &\text{if $n>0$;}\\
\frac{1}{(1-xq^{-1})(1-xq^{-2})\cdots(1-xq^{n})} &\text{if $n<0$.}\\
\end{cases}
\end{equation}
Strictly speaking, the   above $q$-Pochhammer  symbol is not well-defined when $n$ is a negative integer and $a=q^{k}$ for some $1\leq k \leq -n$. However, this is not the case in our paper.

We are now ready to state our first main theorem.



\begin{thm}\label{semithm1}
Assume that $x,y,m,n$ are non-negative integers, and $(a_i)_{i=1}^{m}$ and $(b_j)_{j=1}^{n}$ are two sequences of positive integers between $1$ and $m+n$. If $S_{x}((a_i)_{i=1}^{m}; (b_j)_{j=1}^{n})$ is tile-able, then we always have
\begin{align}\label{maineq1}
\frac{\M(S_{x}((a_i)_{i=1}^{m}; (b_j)_{j=1}^{n}))}{\M(S_{y}((a_i)_{i=1}^{m}; (b_j)_{j=1}^{n}))}&=q^{(y-x)\left(\sum_{i=1}^{m}a_i+\sum_{j=1}^{n}b_j-\frac{(m+n)(m+n+1)}{2}\right)}\frac{\PP_{q^2}(y,m,n)}{\PP_{q^2}(x,m,n)}\notag\\
&\times\prod_{i=1}^{m}\frac{(q^{2(x+i)};q^2)_{a_i-i}}{(q^{2(y+i)};q^2)_{a_i-i}}\prod_{j=1}^{n}\frac{(q^{2(x+j)};q^2)_{b_j-j}}{(q^{2(y+j)};q^2)_{b_j-j}},
\end{align}
where $\PP_q(a,b,c)$ is the generating function of  the plane partitions fitting in an $(a\times b \times c)$-box and given by MacMahon's celebrated formula for boxed plane partitions \cite{Mac}:
\begin{equation}
\PP_q(a,b,c)=\prod_{i=1}^{a}\prod_{j=1}^{b}\prod_{k=1}^{c}\frac{q^{i+j+k-1}-1}{q^{i+j+k-2}-1}.
\end{equation}
\end{thm}
We note that the ratio of generating functions in (\ref{maineq1}) does \emph{not} depend on the indeterminates $X$ and $Y$ in the weight assignment of the region $S_{x}((a_i)_{i=1}^{m}; (b_j)_{j=1}^{n})$.

We now re-assign weights to the right lozenges of the semi-hexagon as in Figure \ref{Fig:weight}(b). In particular, the right lozenge with center at $(i,j)$ is weighted by $\frac{Xq^{i}+Yq^{-i}}{2}$; all the left and vertical lozenges now have weight $1$.  Denote by $S'_{x}((a_i)_{i=1}^{m}; (b_j)_{j=1}^{n})$ the new weighted  region. In other words, $S_x(\textbf{a},\textbf{b})$ and $S'_x(\textbf{a},\textbf{b})$ have the same shape, but different weight assignments. The ratio of the tiling generating functions of $S'_x(\textbf{a},\textbf{b})$ and $S'_y(\textbf{a},\textbf{b})$ is also given  by a simple product formula. It is, in fact, very similar to that in the previous theorem.

\begin{thm}\label{semithm2}
Assume that $x,y,m,n$ are non-negative integers with $y\geq x$, and $(a_i)_{i=1}^{m}$ and $(b_j)_{j=1}^{n}$ are two sequences of positive integers between $1$ and $m+n$. If $S'_{x}((a_i)_{i=1}^{m}; (b_j)_{j=1}^{n})$ is tile-able, then
\begin{align}\label{maineq2}
\frac{\M(S'_{x}((a_i)_{i=1}^{m}; (b_j)_{j=1}^{n}))}{\M(S'_{y}((a_i)_{i=1}^{m}; (b_j)_{j=1}^{n}))}&=q^{\frac n2(y^2-x^2)
   +(y-x)\left(
   \sum_{i=1}^{m}a_i-\frac{(m+n)^2}{2}-\frac m2
   \right)}\frac{\PP_{q^2}(y,m,n)}{\PP_{q^2}(x,m,n)}\notag\\
&\times\prod_{j=1}^{n}\prod_{i=1}^{y-x}\left(\frac{2}{X+q^{2(x+i-b_j)}Y}\right)\prod_{i=1}^{m}\frac{(q^{2(x+i)};q^2)_{a_i-i}}{(q^{2(y+i)};q^2)_{a_i-i}}\prod_{j=1}^{n}\frac{(q^{2(x+j)};q^2)_{b_j-j}}{(q^{2(y+j)};q^2)_{b_j-j}}.
\end{align}
\end{thm}

\begin{figure}\centering
\setlength{\unitlength}{3947sp}%
\begingroup\makeatletter\ifx\SetFigFont\undefined%
\gdef\SetFigFont#1#2#3#4#5{%
  \reset@font\fontsize{#1}{#2pt}%
  \fontfamily{#3}\fontseries{#4}\fontshape{#5}%
  \selectfont}%
\fi\endgroup%
\resizebox{!}{7cm}{
\begin{picture}(0,0)%
\includegraphics{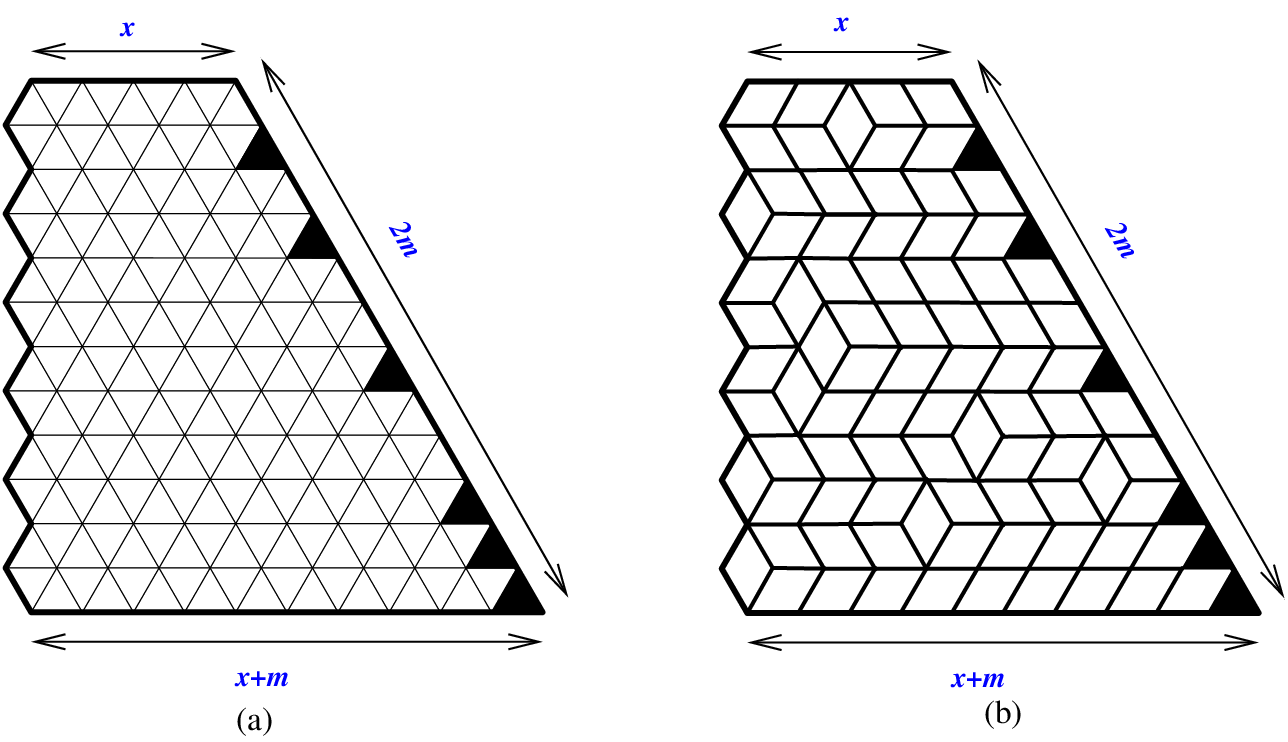}%
\end{picture}%
%
%
\begin{picture}(10296,5933)(775,-5277)
\put(4599,-3781){\makebox(0,0)[lb]{\smash{{\SetFigFont{14}{16.8}{\rmdefault}{\bfdefault}{\updefault}{\color[rgb]{1,1,1}$a_5$}%
}}}}
\put(8903,-1320){\makebox(0,0)[lb]{\smash{{\SetFigFont{14}{16.8}{\rmdefault}{\bfdefault}{\updefault}{\color[rgb]{1,1,1}$a_2$}%
}}}}
\put(9503,-2385){\makebox(0,0)[lb]{\smash{{\SetFigFont{14}{16.8}{\rmdefault}{\bfdefault}{\updefault}{\color[rgb]{1,1,1}$a_3$}%
}}}}
\put(10118,-3435){\makebox(0,0)[lb]{\smash{{\SetFigFont{14}{16.8}{\rmdefault}{\bfdefault}{\updefault}{\color[rgb]{1,1,1}$a_4$}%
}}}}
\put(10328,-3787){\makebox(0,0)[lb]{\smash{{\SetFigFont{14}{16.8}{\rmdefault}{\bfdefault}{\updefault}{\color[rgb]{1,1,1}$a_5$}%
}}}}
\put(10538,-4132){\makebox(0,0)[lb]{\smash{{\SetFigFont{14}{16.8}{\rmdefault}{\bfdefault}{\updefault}{\color[rgb]{1,1,1}$a_6$}%
}}}}
\put(8483,-607){\makebox(0,0)[lb]{\smash{{\SetFigFont{14}{16.8}{\rmdefault}{\bfdefault}{\updefault}{\color[rgb]{1,1,1}$a_1$}%
}}}}
\put(2754,-601){\makebox(0,0)[lb]{\smash{{\SetFigFont{14}{16.8}{\rmdefault}{\bfdefault}{\updefault}{\color[rgb]{1,1,1}$a_1$}%
}}}}
\put(3174,-1314){\makebox(0,0)[lb]{\smash{{\SetFigFont{14}{16.8}{\rmdefault}{\bfdefault}{\updefault}{\color[rgb]{1,1,1}$a_2$}%
}}}}
\put(3774,-2379){\makebox(0,0)[lb]{\smash{{\SetFigFont{14}{16.8}{\rmdefault}{\bfdefault}{\updefault}{\color[rgb]{1,1,1}$a_3$}%
}}}}
\put(4389,-3429){\makebox(0,0)[lb]{\smash{{\SetFigFont{14}{16.8}{\rmdefault}{\bfdefault}{\updefault}{\color[rgb]{1,1,1}$a_4$}%
}}}}
\put(4809,-4126){\makebox(0,0)[lb]{\smash{{\SetFigFont{14}{16.8}{\rmdefault}{\bfdefault}{\updefault}{\color[rgb]{1,1,1}$a_6$}%
}}}}
\end{picture}}
\caption{(a) A quartered with dents on the right side and (b) a tiling of its. The black triangles indicate the unit triangles removed.}\label{Fig:Semitwodent5}
\end{figure}
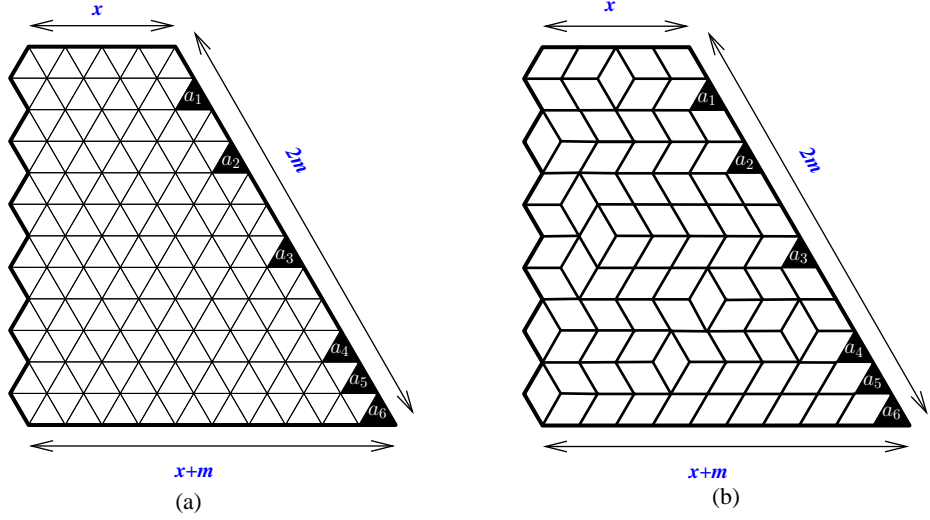

\begin{figure}\centering
\includegraphics[width=13cm]{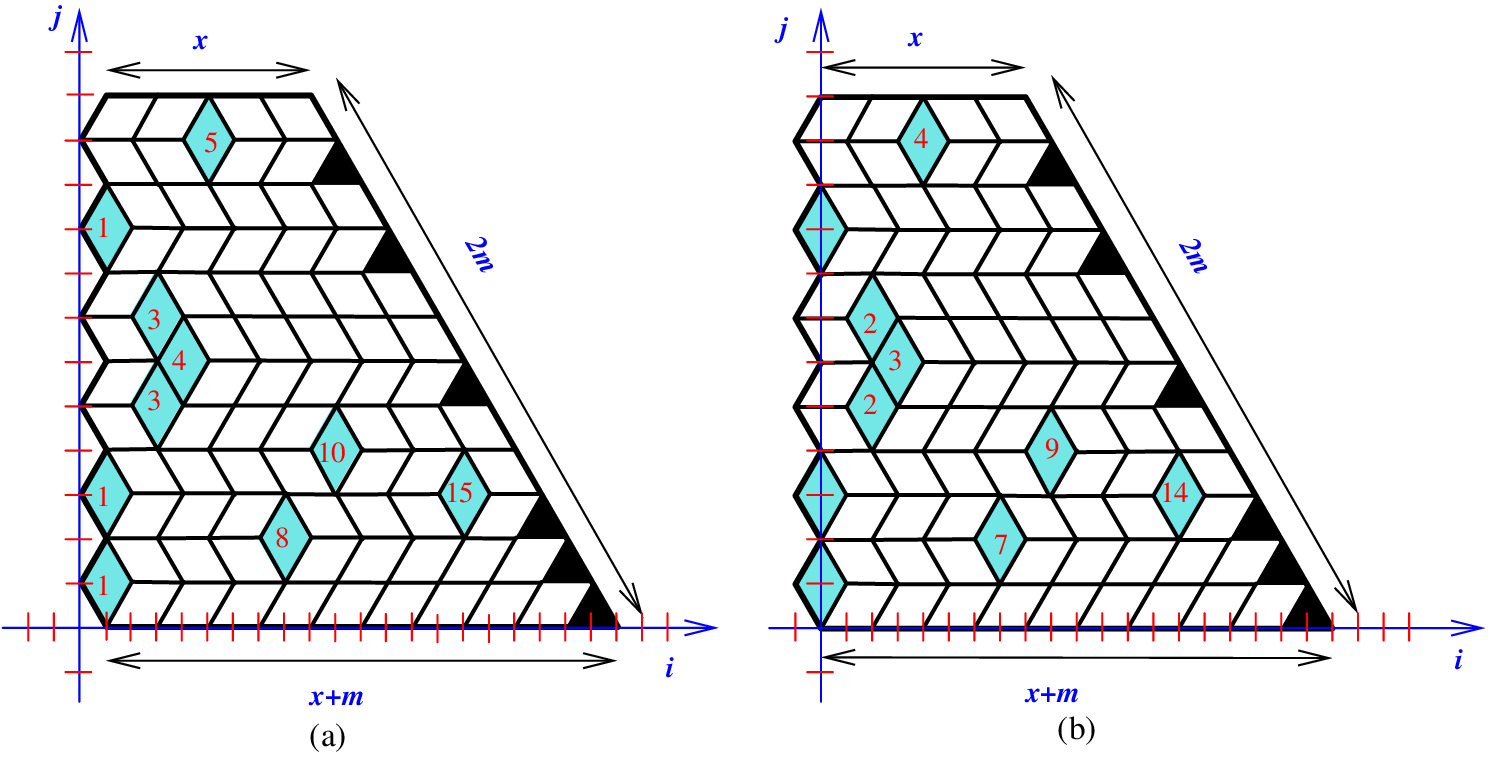}%
\caption{Assigning weights to the lozenges of a quartered hexagon. The shaded lozenges passed by the $j$-axis are weighted by $\frac{1}{2}$; the ones with label $n$ are weighted by  $\frac{q^{n}+q^{-n}}{2}$.}\label{Fig:weight2}
\end{figure}

We consider next a new family of regions called ``\emph{quartered hexagons}" with dents on the right side, as shown in Figure \ref{Fig:Semitwodent5}. In particular, we consider a right trapezoidal region of side-lengths $x,2m,x+m,2m$. The vertical left side runs along a zigzag path with $2m$ steps. We remove $m$ up-pointing unit triangles from the right side of the region at the positions $a_1,a_2,\dots,a_m$, from top to bottom. We also assign the weights to vertical lozenges of the region, as in Figure \ref{Fig:weight2}(a). In particular, the $j$-axis is touching the right side of the region, and the $i$-axis runs along the base. The vertical lozenge with the center at the point $(i,j)$ is weighted by $\frac{q^{i}+q^{-i}}{2}$. (The weight of a tiling is still the product of its lozenge-weights as usual.) Denote by $Q_{x}((a_i)_{i=1}^{m})$ this weighted region. Similar to the case of Theorems \ref{semithm1} and \ref{semithm2}, the ratio of tiling generating functions of 
$Q_{x}((a_i)_{i=1}^{m})$ and its sibling $Q_{y}((a_i)_{i=1}^{m})$ is always given by a simple product formula (even though, each tilling generating function is not a simple product).

\begin{thm}\label{halfthm1}
For non-negative integers $x,y,m$ and a sequence $\textbf{a}=(a_i)_{i=1}^{m}$ of positive integers between $1$ and $2m$, we have
\begin{align}\label{halfeq1}
\frac{\M(Q_{x}((a_i)_{i=1}^{m}))}{\M(Q_{y}((a_i)_{i=1}^{m}))}&=q^{2(y-x)(\sum_{i=1}^{m}a_i- m^2)}\prod_{i=1}^{m}\frac{(q^{2(2y+a_i+1)};q^2)_{2i-a_i-1}}{(q^{2(2x+a_i+1)};q^2)_{2i-a_i-1}}
\end{align}
whenever $Q_{x}((a_i)_{i=1}^{m})$ is tile-able.
\end{thm}

Next, we consider a variation of the quartered region above. The new weighted region has the same shape as the one in Theorem \ref{halfthm1}. The only difference is in the lozenge-weights. We now re-assign the weights to the lozenges, as in Figure \ref{Fig:weight2}(b).  One should note that the $j$-axis is now slightly to the right of that in Figure \ref{Fig:weight2}(a). As a consequence, our region now has some vertical lozenges intersected by the $j$-axis. The vertical lozenge with center at the point $(i,j)$ in the new coordinate system is still weighted by $\frac{q^{i}+q^{-i}}{2}$, with one exception: the vertical lozenges intersected by the $j$-axis are weighted by $1/2$ (\emph{not} by $1=\frac{q^{0}+q^{-0}}{2}$). Denote by $Q'_{x}((a_i)_{i=1}^{m})$ the new weighted region.

\begin{thm}\label{halfthm2}
Assume that $x,y,m$ are non-negative integers and that $(a_i)_{i=1}^{m}$ is a sequence of positive integers between $1$ and $2m$. Then
\begin{align}\label{halfeq2}
\frac{\M(Q'_{x}((a_i)_{i=1}^{m}))}{\M(Q'_{y}((a_i)_{i=1}^{m}))}&=q^{2(y-x)(\sum_{i=1}^{m}a_i- m^2)}\prod_{i=1}^{m}\frac{(q^{2(2y+a_i)};q^2)_{2i-a_i-1}}{(q^{2(2x+a_i)};q^2)_{2i-a_i-1}}
\end{align}
if $Q'_{x}((a_i)_{i=1}^{m})$ is tile-able.
\end{thm}

\begin{rmk}[Combinatorial reciprocity phenomenon]\label{reciprocity}
The ratio in Theorem \ref{halfthm2} is obtained from the one in Theorem \ref{halfthm1} by replacing $x$ by $x-1/2$ and $y$ by $y-1/2$. This reminds us to the ``\emph{combinatorial reciprocity phenomenon}": even though the regions 
$Q_{x}((a_i)_{i=1}^{m})$ and $Q_{y}((a_i)_{i=1}^{m})$ are \emph{not} defined when $x$ and $y$  are half integers, the formula of the ratio of  their tiling generating functions gives the  ``numbers" of combinatorial objects of a different sort when evaluated at half-integers. It would be interesting to find a direct explanation for this, i.e., an explanation without requiring the calculation of the tiling generating functions.  We refer the reader to, e.g,  \cite{Beck,StanleyRecip,ProppRecip} for more discussions about the combinatorial reciprocity phenomenon.
\end{rmk}

\begin{rmk} It is worth noticing that Fulmek recently shows that one could obtain the results in Theorems \ref{semithm1} and \ref{halfthm1} by using lattice path combinatorics and a sepcial matrix factorization \cite{Ful2}. \end{rmk}

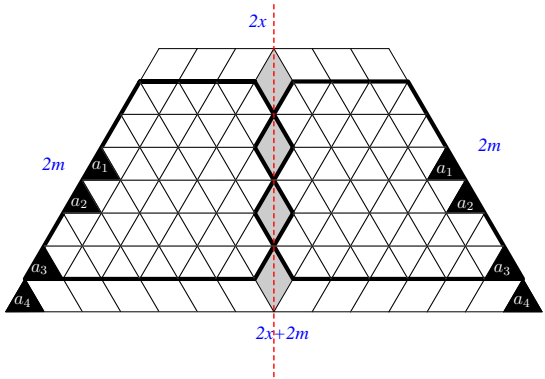
\begin{figure}\centering
\setlength{\unitlength}{3947sp}%
\begingroup\makeatletter\ifx\SetFigFont\undefined%
\gdef\SetFigFont#1#2#3#4#5{%
  \reset@font\fontsize{#1}{#2pt}%
  \fontfamily{#3}\fontseries{#4}\fontshape{#5}%
  \selectfont}%
\fi\endgroup%
\resizebox{!}{5cm}{
\begin{picture}(0,0)%
\includegraphics{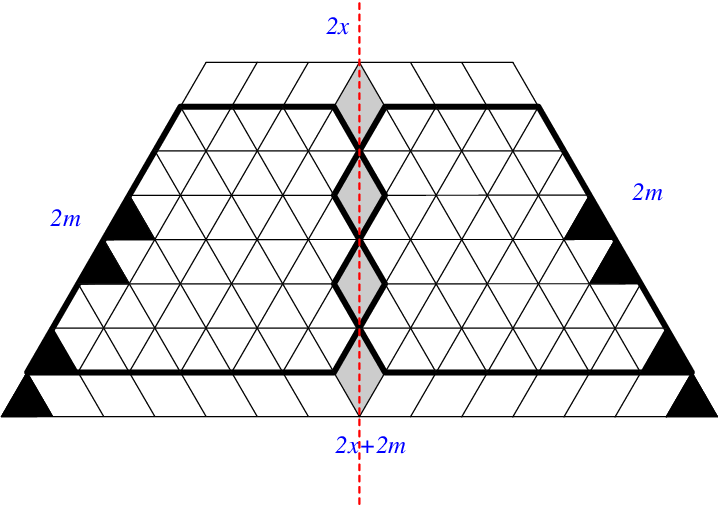}%
\end{picture}%
%
%

\begin{picture}(5802,4060)(807,-4380)
\put(909,-3571){\makebox(0,0)[lb]{\smash{{\SetFigFont{12}{14.4}{\rmdefault}{\mddefault}{\itdefault}{\color[rgb]{1,1,1}$a_4$}%
}}}}
\put(6241,-3564){\makebox(0,0)[lb]{\smash{{\SetFigFont{12}{14.4}{\rmdefault}{\mddefault}{\itdefault}{\color[rgb]{1,1,1}$a_4$}%
}}}}
\put(1740,-2114){\makebox(0,0)[lb]{\smash{{\SetFigFont{12}{14.4}{\rmdefault}{\mddefault}{\itdefault}{\color[rgb]{1,1,1}$a_1$}%
}}}}
\put(5420,-2154){\makebox(0,0)[lb]{\smash{{\SetFigFont{12}{14.4}{\rmdefault}{\mddefault}{\itdefault}{\color[rgb]{1,1,1}$a_1$}%
}}}}
\put(1516,-2506){\makebox(0,0)[lb]{\smash{{\SetFigFont{12}{14.4}{\rmdefault}{\mddefault}{\itdefault}{\color[rgb]{1,1,1}$a_2$}%
}}}}
\put(5641,-2529){\makebox(0,0)[lb]{\smash{{\SetFigFont{12}{14.4}{\rmdefault}{\mddefault}{\itdefault}{\color[rgb]{1,1,1}$a_2$}%
}}}}
\put(1089,-3219){\makebox(0,0)[lb]{\smash{{\SetFigFont{12}{14.4}{\rmdefault}{\mddefault}{\itdefault}{\color[rgb]{1,1,1}$a_3$}%
}}}}
\put(6031,-3226){\makebox(0,0)[lb]{\smash{{\SetFigFont{12}{14.4}{\rmdefault}{\mddefault}{\itdefault}{\color[rgb]{1,1,1}$a_3$}%
}}}}
\end{picture}}
\caption{Separating a symmetric semi-hexagon into two congruent quartered hexagons.}\label{Fig:sym}
\end{figure}

Theorem \ref{halfthm1} implies the following symmetric version of Theorem \ref{semithm1}. We consider the symmetric  weighted semi-hexagon $S_{2x}(\textbf{a};\textbf{a})$ with $X=Y=1$. Denote by $\M_s(S_{2x}(\textbf{a};\textbf{a}))$ the sum of the \emph{square roots} of weights of all reflectively symmetric tilings of $S_{2x}(\textbf{a};\textbf{a})$. Assume that $S_{2x}(\textbf{a};\textbf{a})$ admits a reflective symmetric tiling, it is easy to see that, say  by Lemma \ref{tileability1}, we must have $a_1>1$ and $a_m=2m$. 
\begin{cor}[Symmetric version of Theorem \ref{semithm1}] Assume that $x,y,m$ are non-negative integers and that $\textbf{a}=(a_i)_{i=1}^{m}$ is a sequence of positive integers between $1$ and $2m$. Assume besides that $a_1>1$ and $a_m=2m$. Then
\begin{align}\label{coreq}
\frac{\M_s(S_{2x}(\textbf{a};\textbf{a}))}{\M_s(S_{2y}(\textbf{a};\textbf{a}))}&=\frac{\M(Q_{x}(a_1-1,a_2-1,\dots,a_{m-1}-1))}{\M(Q_{y}(a_1-1,a_2-1,\dots,a_{m-1}-1))} \notag\\
&=q^{2(y-x)(\sum_{i=1}^{m-1}a_i- m(m-1))}\prod_{i=1}^{m-1}\frac{(q^{2(2y+a_i)};q^2)_{2i-a_i}}{(q^{2(2x+a_i)};q^2)_{2i-a_i}}
\end{align}
if $Q_{x}(a_1-1,a_2-1,\dots,a_{m-1}-1)$ is tile-able.
\end{cor}
\begin{proof}
Each reflectively symmetric tiling of $S_{2x}(\textbf{a};\textbf{a})$ contains $m$ vertical lozenges intersected by the symmetry axis (see the shaded lozenges in Figure \ref{Fig:sym}). Removal of these shaded lozenges separates $S_{2x}(\textbf{a};\textbf{a})$ into two congruent regions (that are the reflection of each other over the symmetry axis). Each of these regions in turn has forced lozenges on the top and bottom rows. Removing of these forced lozenges, we get two copies of the weighted quartered hexagon $Q_{x}(a_1-1,a_2-1,\dots,a_{m-1}-1)$ (indicated by the two regions that are restricted by the bold contours). As the weight of a symmetric tiling in the sum $\M_s(S_{2x}(\textbf{a};\textbf{a}))$ is equal to the square root of the original tiling-weight of  $S_{2x}(\textbf{a};\textbf{a})$, we have a weight-preserving bijection between symmetric tilings of $S_{2x}(\textbf{a};\textbf{a})$ and tilings of the quartered hexagon $Q_{x}(a_1-1,a_2-1,\dots,a_{m-1}-1)$. It means that
\[\M_s(S_{2x}(\textbf{a};\textbf{a}))=\M(Q_{x}(a_1-1,a_2-1,\dots,a_{m-1}-1)).\]
Similarly, we have 
\[\M_s(S_{2y}(\textbf{a};\textbf{a}))=\M(Q_{y}(a_1-1,a_2-1,\dots,a_{m-1}-1)).\]
This implies the first  equality in (\ref{coreq}). The second equality follows from Theorem \ref{halfthm1}.
\end{proof}

\section{Preliminaries}\label{Sec:Pre}

\subsection{Tile-ability}
As mentioned in the previous section, the semi-hexagons with dents on both sides may not have any tilings in general. It is not hard to prove the following tile-ability for these semi-hexagons, based on the correspondence between lozenge tilings and non-intersecting lattice paths.
\begin{lem}\label{tileability1}
Assume that $x,m,n$ are non-negative integers, and $(a_i)_{i=1}^{m}$ and $(b_j)_{j=1}^{n}$ are two sequences of positive integers between $1$ and $m+n$. Then
$S_{x}((a_i)_{i=1}^{m}; (b_j)_{j=1}^{n})$ is tile-able if and only if 
\begin{equation}
|\{a_i\}_{i=1}^{m}\cap[t]|+|\{b_j\}_{j=1}^{n}\cap[t]|\leq t,
\end{equation}
for any positive integer $t$ not excess $m+n$, where we use the notation $[t]$ for the set of the first $t$ positive integers $\{1,2,\dots,t\}$, and where $|A|$ denotes the cardinality of a finite set $A$.
\end{lem}

\begin{figure}\centering
\setlength{\unitlength}{3947sp}%
\begingroup\makeatletter\ifx\SetFigFont\undefined%
\gdef\SetFigFont#1#2#3#4#5{%
  \reset@font\fontsize{#1}{#2pt}%
  \fontfamily{#3}\fontseries{#4}\fontshape{#5}%
  \selectfont}%
\fi\endgroup%
\resizebox{!}{5cm}{
\begin{picture}(0,0)%
\includegraphics{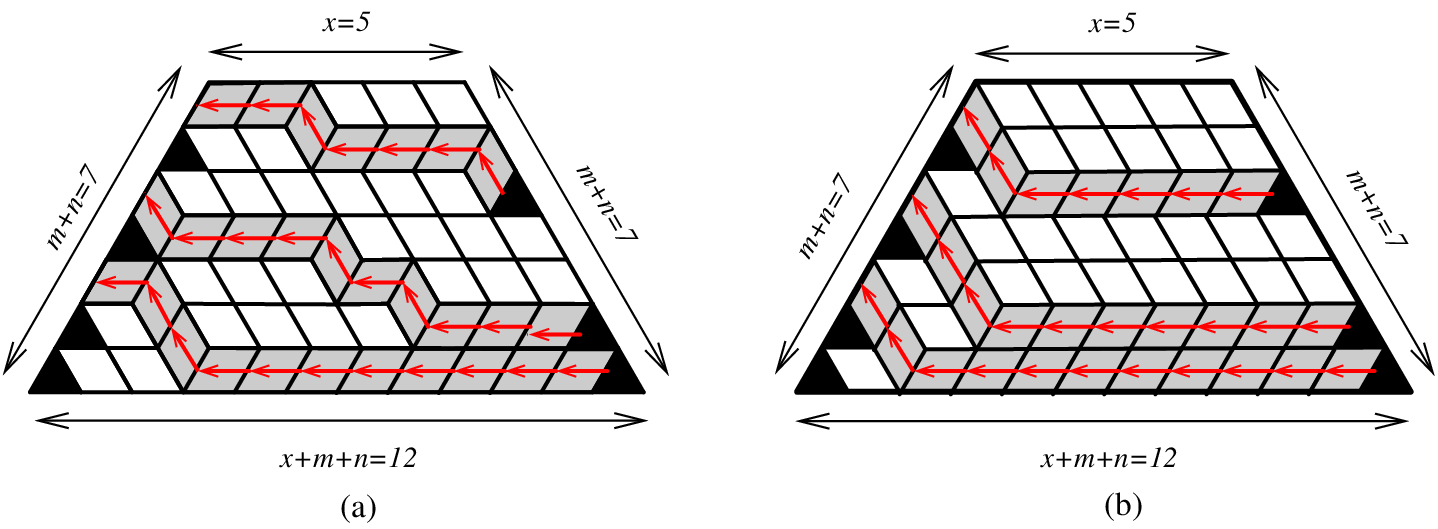}%
\end{picture}%
%
%

\begin{picture}(11500,4233)(576,-3787)
\put(11339,-2245){\makebox(0,0)[lb]{\smash{{\SetFigFont{14}{16.8}{\rmdefault}{\mddefault}{\itdefault}{\color[rgb]{1,1,1}$b_2$}%
}}}}
\put(5182,-2256){\makebox(0,0)[lb]{\smash{{\SetFigFont{14}{16.8}{\rmdefault}{\mddefault}{\itdefault}{\color[rgb]{1,1,1}$b_2$}%
}}}}
\put(5384,-2624){\makebox(0,0)[lb]{\smash{{\SetFigFont{14}{16.8}{\rmdefault}{\mddefault}{\itdefault}{\color[rgb]{1,1,1}$b_3$}%
}}}}
\put(1904,-846){\makebox(0,0)[lb]{\smash{{\SetFigFont{14}{16.8}{\rmdefault}{\mddefault}{\itdefault}{\color[rgb]{1,1,1}$a_1$}%
}}}}
\put(1514,-1566){\makebox(0,0)[lb]{\smash{{\SetFigFont{14}{16.8}{\rmdefault}{\mddefault}{\itdefault}{\color[rgb]{1,1,1}$a_2$}%
}}}}
\put(1087,-2279){\makebox(0,0)[lb]{\smash{{\SetFigFont{14}{16.8}{\rmdefault}{\mddefault}{\itdefault}{\color[rgb]{1,1,1}$a_3$}%
}}}}
\put(877,-2624){\makebox(0,0)[lb]{\smash{{\SetFigFont{14}{16.8}{\rmdefault}{\mddefault}{\itdefault}{\color[rgb]{1,1,1}$a_4$}%
}}}}
\put(4567,-1199){\makebox(0,0)[lb]{\smash{{\SetFigFont{14}{16.8}{\rmdefault}{\mddefault}{\itdefault}{\color[rgb]{1,1,1}$b_1$}%
}}}}
\put(8054,-820){\makebox(0,0)[lb]{\smash{{\SetFigFont{14}{16.8}{\rmdefault}{\mddefault}{\itdefault}{\color[rgb]{1,1,1}$a_1$}%
}}}}
\put(7627,-1547){\makebox(0,0)[lb]{\smash{{\SetFigFont{14}{16.8}{\rmdefault}{\mddefault}{\itdefault}{\color[rgb]{1,1,1}$a_2$}%
}}}}
\put(7222,-2252){\makebox(0,0)[lb]{\smash{{\SetFigFont{14}{16.8}{\rmdefault}{\mddefault}{\itdefault}{\color[rgb]{1,1,1}$a_3$}%
}}}}
\put(7019,-2590){\makebox(0,0)[lb]{\smash{{\SetFigFont{14}{16.8}{\rmdefault}{\mddefault}{\itdefault}{\color[rgb]{1,1,1}$a_4$}%
}}}}
\put(10732,-1172){\makebox(0,0)[lb]{\smash{{\SetFigFont{14}{16.8}{\rmdefault}{\mddefault}{\itdefault}{\color[rgb]{1,1,1}$b_1$}%
}}}}
\put(11557,-2612){\makebox(0,0)[lb]{\smash{{\SetFigFont{14}{16.8}{\rmdefault}{\mddefault}{\itdefault}{\color[rgb]{1,1,1}$b_3$}%
}}}}
\end{picture}}
\caption{Encoding each tiling of the semi-hexagon $S_{x}(\textbf{a},\textbf{b})$ as a family of disjoint lozenge paths.}\label{Fig:Semitwodent2}
\end{figure}
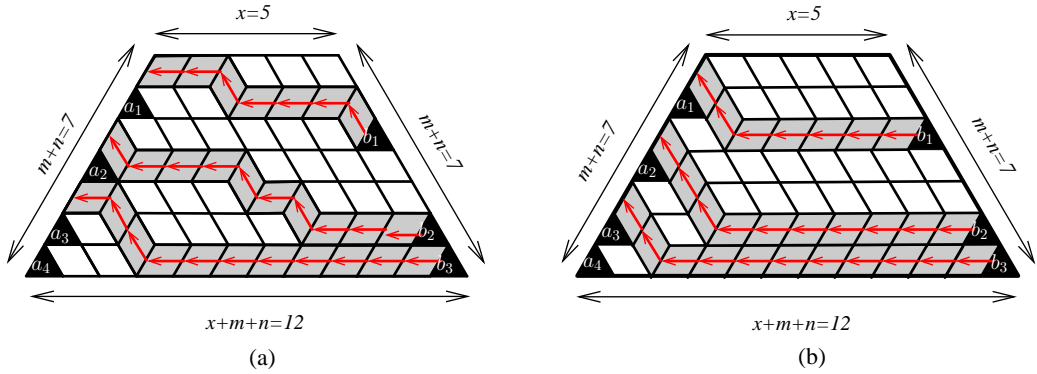

\begin{proof}
It is easy to see that the lemma holds for the case $m=n=0$. Without loss of generality, we assume in the rest of the proof that $n>0$. 

Assume that the semi-hexagon $S_{x}((a_i)_{i=1}^m; (b_{j})_{j=1}^{n})$ is tile-able. Each tiling of it can be encoded uniquely as a family of $n$ disjoint paths of juxtaposing lozenges, as shown by the shaded paths in Figure \ref{Fig:Semitwodent2}. In particular, each of the lozenge paths consists of right and vertical lozenges. The path goes from a dent position on the right side to a non-dent position on the left side. The lozenges outside these $m$ paths are all left lozenges. As the path always goes weakly upward, the ending position of a path is not lower than the starting position.

Let $t$ be any positive integer in $[m+n]$. Assume that $|\{b_j\}_{i=1}^{n}\cap [t]|=p$, i.e., there are exactly $p$  `\emph{right dents}' (dents on the right side) within the distance $t$ from the top of the semi-hexagon. Each of these  right dents is connected to a non-dent position within the distance $t$ on the left side by a lozenge path. It means that these lozenge paths yield an injective mapping from the set $\{b_j\}_{i=1}^{n}\cap [t]$ to the set $[t]-\{a_i\}_{i=1}^{m}$. It implies that \[|\{b_j\}_{i=1}^{n}\cap [t]|\leq t- |\{a_i\}_{i=1}^{m}\cap [t]|,\] or $|\{b_j\}_{i=1}^{n}\cap [t]|+|\{a_i\}_{i=1}^{m}\cap [t]|\leq t$ as desired. 

In reverse, assume that $|\{b_j\}_{i=1}^{n}\cap [t]|+|\{a_i\}_{i=1}^{m}\cap [t]|\leq t$ for any $t\in [m+n]$. We need to point out a particular tiling of the region $S_{x}((a_i)_{i=1}^n; (b_{j})_{j=1}^{n})$. It is easy to see that, in this case, the region always has a tiling as shown in Figure \ref{Fig:Semitwodent2}(b). Each lozenge path is now a shaded ``hook." This finishes the proof of the lemma.
\end{proof}

Similar to the case of the semi-hexagons, the quartered hexagon $Q_{x}((a_i)_{i=1}^{m})$ may have no tiling. The following lemma provides a condition for the tile-ability of this region.
\begin{lem}\label{tileability2}
Assume that $x,m$ are non-negative integers and that $(a_i)_{i=1}^{m}$ is a sequence of positive integers between $1$ and $2m$. Then
$Q_{x}((a_i)_{i=1}^{m})$ is tile-able if and only if 
\begin{equation}
|\{a_i\}_{i=1}^{m}\cap[2t]|\leq t,
\end{equation}
for any positive integer $t$ not excess $m$.
\end{lem}

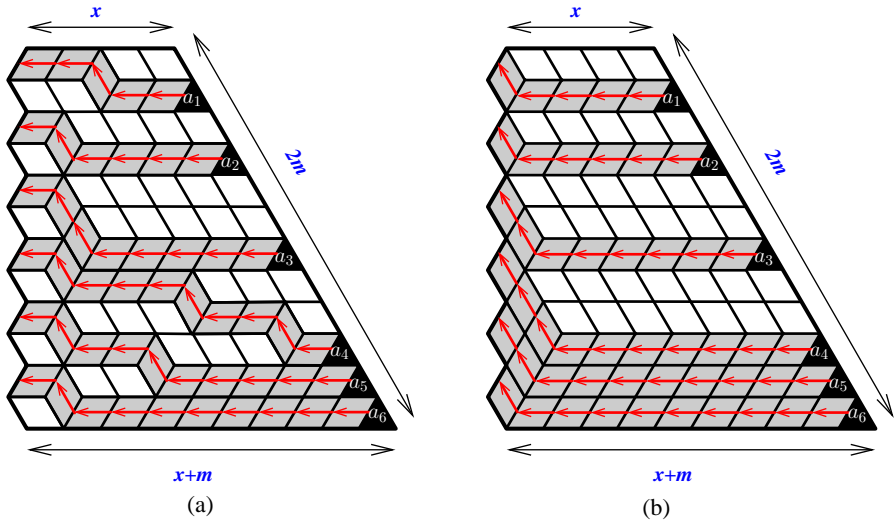
\begin{figure}\centering
\setlength{\unitlength}{3947sp}%
\begingroup\makeatletter\ifx\SetFigFont\undefined%
\gdef\SetFigFont#1#2#3#4#5{%
  \reset@font\fontsize{#1}{#2pt}%
  \fontfamily{#3}\fontseries{#4}\fontshape{#5}%
  \selectfont}%
\fi\endgroup%
\resizebox{!}{7cm}{
\begin{picture}(0,0)%
\includegraphics{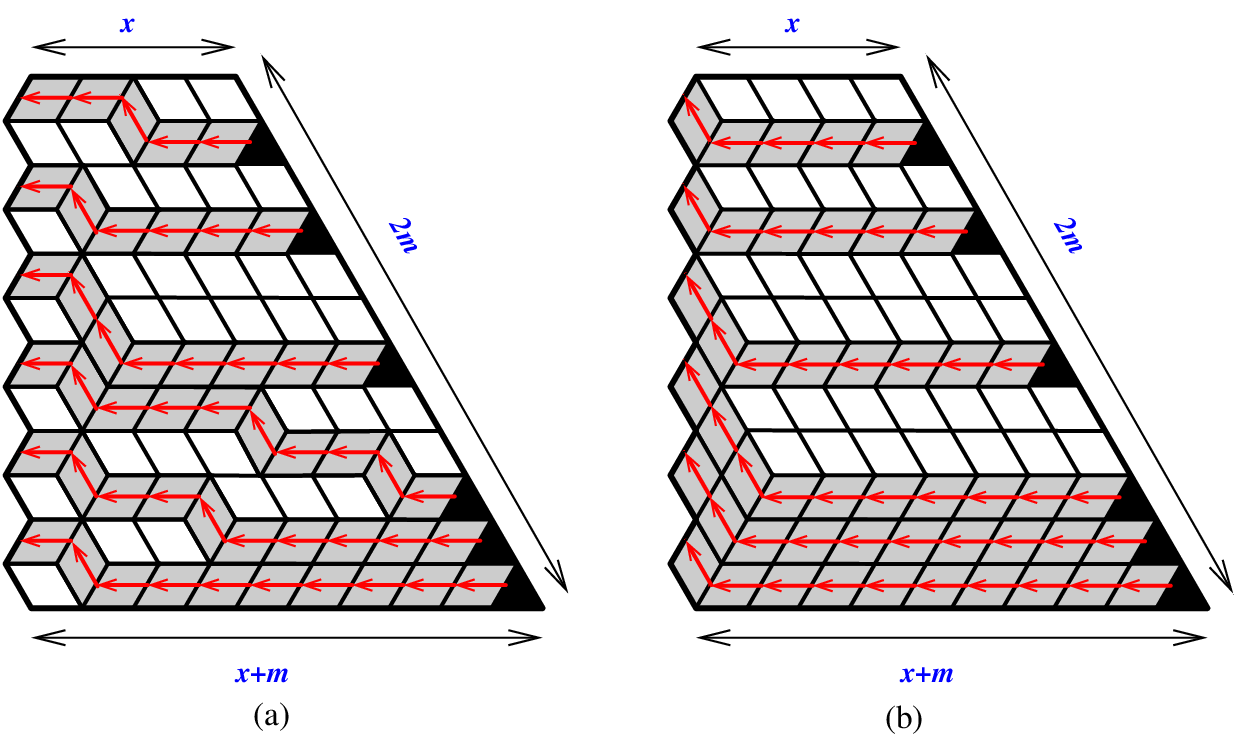}%
\end{picture}%
%
%

\begin{picture}(9886,5902)(775,-5295)
\put(9918,-3781){\makebox(0,0)[lb]{\smash{{\SetFigFont{14}{16.8}{\rmdefault}{\bfdefault}{\updefault}{\color[rgb]{1,1,1}$a_5$}%
}}}}
\put(10128,-4126){\makebox(0,0)[lb]{\smash{{\SetFigFont{14}{16.8}{\rmdefault}{\bfdefault}{\updefault}{\color[rgb]{1,1,1}$a_6$}%
}}}}
\put(8073,-601){\makebox(0,0)[lb]{\smash{{\SetFigFont{14}{16.8}{\rmdefault}{\bfdefault}{\updefault}{\color[rgb]{1,1,1}$a_1$}%
}}}}
\put(2754,-601){\makebox(0,0)[lb]{\smash{{\SetFigFont{14}{16.8}{\rmdefault}{\bfdefault}{\updefault}{\color[rgb]{1,1,1}$a_1$}%
}}}}
\put(3174,-1314){\makebox(0,0)[lb]{\smash{{\SetFigFont{14}{16.8}{\rmdefault}{\bfdefault}{\updefault}{\color[rgb]{1,1,1}$a_2$}%
}}}}
\put(3774,-2379){\makebox(0,0)[lb]{\smash{{\SetFigFont{14}{16.8}{\rmdefault}{\bfdefault}{\updefault}{\color[rgb]{1,1,1}$a_3$}%
}}}}
\put(4389,-3429){\makebox(0,0)[lb]{\smash{{\SetFigFont{14}{16.8}{\rmdefault}{\bfdefault}{\updefault}{\color[rgb]{1,1,1}$a_4$}%
}}}}
\put(4599,-3781){\makebox(0,0)[lb]{\smash{{\SetFigFont{14}{16.8}{\rmdefault}{\bfdefault}{\updefault}{\color[rgb]{1,1,1}$a_5$}%
}}}}
\put(4809,-4126){\makebox(0,0)[lb]{\smash{{\SetFigFont{14}{16.8}{\rmdefault}{\bfdefault}{\updefault}{\color[rgb]{1,1,1}$a_6$}%
}}}}
\put(8493,-1314){\makebox(0,0)[lb]{\smash{{\SetFigFont{14}{16.8}{\rmdefault}{\bfdefault}{\updefault}{\color[rgb]{1,1,1}$a_2$}%
}}}}
\put(9093,-2379){\makebox(0,0)[lb]{\smash{{\SetFigFont{14}{16.8}{\rmdefault}{\bfdefault}{\updefault}{\color[rgb]{1,1,1}$a_3$}%
}}}}
\put(9708,-3429){\makebox(0,0)[lb]{\smash{{\SetFigFont{14}{16.8}{\rmdefault}{\bfdefault}{\updefault}{\color[rgb]{1,1,1}$a_4$}%
}}}}
\end{picture}}
\caption{Encoding each tiling of the quartered hexagon $Q_{x}(\textbf{a})$ as a family of disjoint lozenge paths.}\label{Fig:Semitwodent6}
\end{figure}

\begin{proof}
This lemma can be proved similarly to Lemma \ref{tileability1} above. Each tiling of the halved hexagon $Q_{x}((a_i)_{i=1}^{m})$ (if exist) can be encoded as a family of $m$ disjoint lozenge paths. These paths go from a dent on the right side to an odd step on the region's vertical zigzag side (see Figure \ref{Fig:Semitwodent6}). We also note that all the lozenge paths go weakly upward. For each $t\in [m]$, we consider the dents within the distance $2t$ from the top of the region. Each of these dents is connected to one of $t$ odd steps between $1$ and $2t$ (as these lozenge paths go weakly upward). It means that the number of dents within the distance $2t$ from the top is at most $t$. Equivalently, we have $|\{a_i\}_{i=1}^{m}\cap[2t]|\leq t.$

Reversely, if we have $|\{a_i\}_{i=1}^{m}\cap[2t]|\leq t$, for any $t\in [m]$, then the quartered hexagon $Q_{x}((a_i)_{i=1}^{m})$ always has a tiling as shown in Figure \ref{Fig:Semitwodent6}(b) (the corresponding lozenge paths are indicated by the $m$ shaded hooks). This completes the proof.
\end{proof}

\subsection{Kuo Condensations and Region-splitting Lemma} In the early $2000$s, Eric H. Kuo \cite{Kuo} proved  several combinatorial interpretations of  the well-known Dodgson condensation  in linear algebra \cite{Dodgson}.  Kuo condensation has become a powerful tool in the field of  enumeration of tilings.

A \emph{perfect matching} of a simple graph is a collection of disjoint edges that cover all vertices of the graph. We use the notation $\MM(G)$ for the weighted sum of the perfect matchings of the weighted graph $G$, where the weight of a perfect matching is the product of the weights of its edges. We call $\MM(G)$ the \emph{matching generating function} of $G$. There is a one-to-one correspondence between tilings of a region $R$ on the triangular lattice and perfect matchings of its \emph{(planar) dual graph} $G$ (i.e., the graph whose vertices are the unit triangles in $R$ and whose edges connect precisely two unit triangles sharing an edge). Each edge of the  dual graph inherits the weight of the corresponding lozenge in the region. In particular, we have $\M(R)=\MM(G).$

We will employ the following three versions of the Kuo condensation in our proofs.

\begin{lem}[Theorem 5.1 in \cite{Kuo}]\label{kuothm0}
Let $G=(V_1,V_2,E)$ be a weighted plane bipartite graph in which $|V_1|=|V_2|$. Let vertices $u,v,w,s$ appear on a face of $G$, in that order. If $u,w \in V_1$ and $v,s\in V_2$, then
\begin{equation}\label{kuoeq0}
\MM(G)\MM(G-\{u,v,w,s\})=\MM(G-\{u,v\})\MM(G-\{w,s\})+\MM(G-\{u,s\})\MM(G-\{v,w\}).
\end{equation}
\end{lem}

\begin{lem}[Theorem 5.3 in \cite{Kuo}]\label{kuothm1}
Let $G=(V_1,V_2,E)$ be a weighted plane bipartite graph in which $|V_1|=|V_2|+1$. Let vertices $u,v,w,s$ appear on a face of $G$, in that order. If $u,v,w \in V_1$ and $s\in V_2$, then
\begin{equation}\label{kuoeq1}
\MM(G-\{v\})\MM(G-\{u,w,s\})=\MM(G-\{u\})\MM(G-\{v,w,s\})+\MM(G-\{w\})\MM(G-\{u,v,s\}).
\end{equation}
\end{lem}

\begin{lem}[Theorem 5.4 in \cite{Kuo}]\label{kuothm2}
Let $G=(V_1,V_2,E)$ be a weighted plane bipartite graph in which $|V_1|=|V_2|+2$. Let vertices $u,v,w,s$ appear on a face of $G$, in that order. If $u,v,w,s \in V_1$, then
\begin{equation}\label{kuoeq2}
\MM(G-\{u,w\})\MM(G-\{v,s\})=\MM(G-\{u,v\})\MM(G-\{w,s\})+\MM(G-\{u,s\})\MM(G-\{v,w\}).
\end{equation}
\end{lem}

A \emph{forced lozenge} of the region $R$ is a lozenge contained in any tilings of $R$. Assume that we remove forced lozenges $l_1,l_2,\dots,l_k$ from $R$ and get a new region $R'$, then we have
\begin{equation}
\M(R)=\left(\prod_{i=1}^{k}\wt(l_i)\right) \cdot \M(R'),
\end{equation}
where $\wt(l_i)$ is the weight of the removed lozenge $l_i$.

A region in the triangular lattice\footnote{We only consider regions in the triangular lattice in this paper. From now on, we will use the term``region(s)" to mean ``region(s) in the triangular lattice".} must have  the same number of up-pointing and down-pointing  unit triangles  to admit  a tiling.  We call such a region \emph{balanced}.  The following simple lemma is especially useful when enumerating tilings.

\begin{lem}[Region-splitting Lemma \cite{Tri19,Tri18}]\label{RS}
Assume $R$ is a balanced region  and $Q$  is a subregion of $R$ satisfying two conditions:
\begin{enumerate}
\item The unit triangles in $Q$ lying on the  boundary between $Q$ and $R\setminus Q$ have the same orientation (all are up-pointing or all are down-pointing);
\item $Q$ is balanced.
\end{enumerate}
Then we have $\M(R)=\M(Q) \cdot \M(R\setminus Q)$.
\end{lem}

\subsection{Four basic enumerations} 

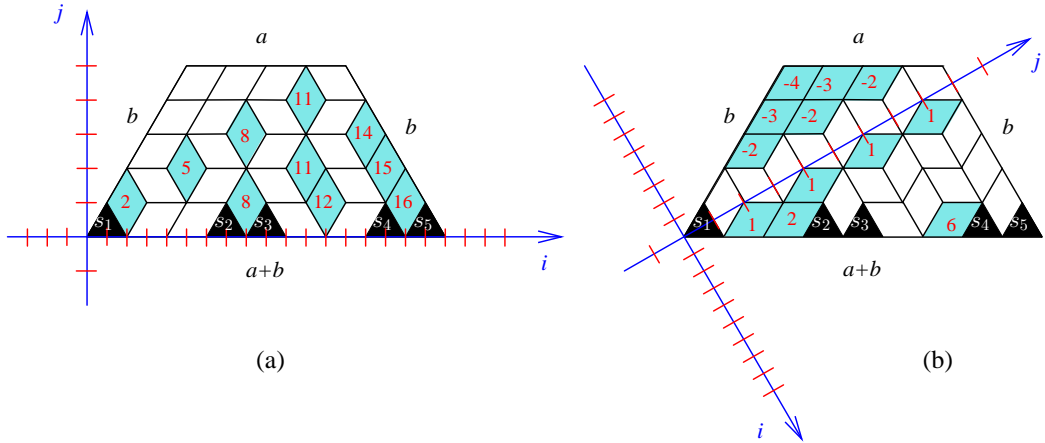
\begin{figure}
\setlength{\unitlength}{3947sp}%
\begingroup\makeatletter\ifx\SetFigFont\undefined%
\gdef\SetFigFont#1#2#3#4#5{%
  \reset@font\fontsize{#1}{#2pt}%
  \fontfamily{#3}\fontseries{#4}\fontshape{#5}%
  \selectfont}%
\fi\endgroup%
\resizebox{!}{6cm}{
\begin{picture}(0,0)%
\includegraphics{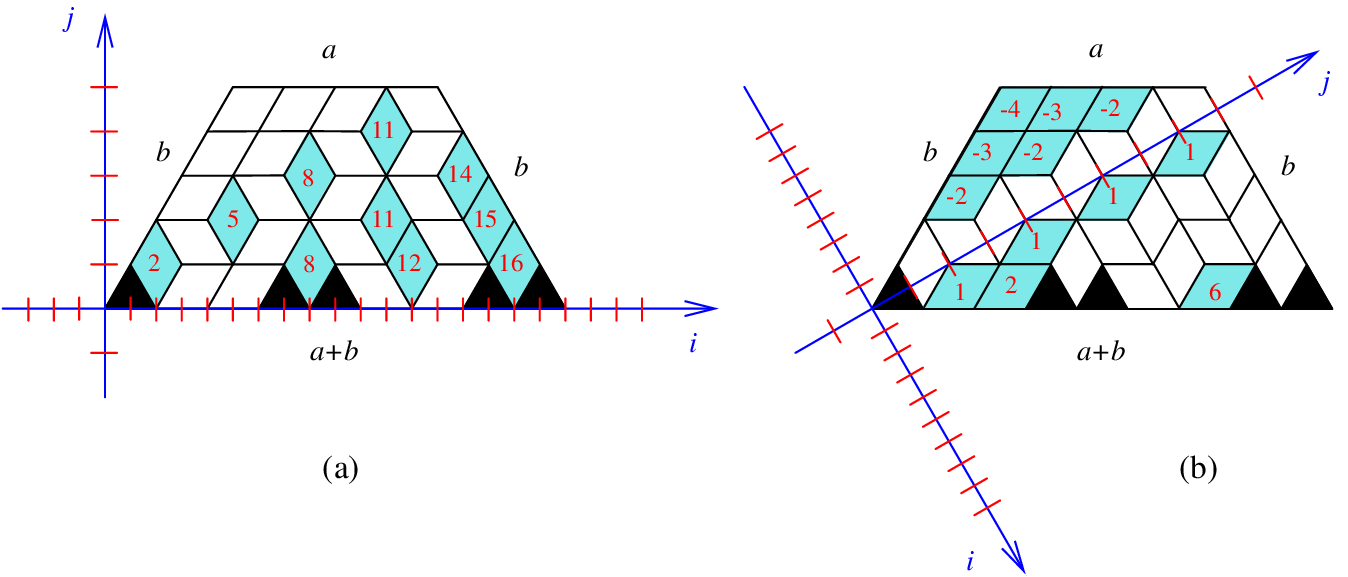}%
\end{picture}%
%
%
\begin{picture}(10734,4690)(388,-4431)
\put(1299,-2180){\makebox(0,0)[lb]{\smash{{\SetFigFont{14}{16.8}{\rmdefault}{\mddefault}{\itdefault}{\color[rgb]{1,1,1}$s_1$}%
}}}}
\put(2536,-2180){\makebox(0,0)[lb]{\smash{{\SetFigFont{14}{16.8}{\rmdefault}{\mddefault}{\itdefault}{\color[rgb]{1,1,1}$s_2$}%
}}}}
\put(2941,-2180){\makebox(0,0)[lb]{\smash{{\SetFigFont{14}{16.8}{\rmdefault}{\mddefault}{\itdefault}{\color[rgb]{1,1,1}$s_3$}%
}}}}
\put(4584,-2180){\makebox(0,0)[lb]{\smash{{\SetFigFont{14}{16.8}{\rmdefault}{\mddefault}{\itdefault}{\color[rgb]{1,1,1}$s_5$}%
}}}}
\put(4164,-2180){\makebox(0,0)[lb]{\smash{{\SetFigFont{14}{16.8}{\rmdefault}{\mddefault}{\itdefault}{\color[rgb]{1,1,1}$s_4$}%
}}}}
\put(7453,-2180){\makebox(0,0)[lb]{\smash{{\SetFigFont{14}{16.8}{\rmdefault}{\mddefault}{\itdefault}{\color[rgb]{1,1,1}$s_1$}%
}}}}
\put(8675,-2180){\makebox(0,0)[lb]{\smash{{\SetFigFont{14}{16.8}{\rmdefault}{\mddefault}{\itdefault}{\color[rgb]{1,1,1}$s_2$}%
}}}}
\put(9073,-2180){\makebox(0,0)[lb]{\smash{{\SetFigFont{14}{16.8}{\rmdefault}{\mddefault}{\itdefault}{\color[rgb]{1,1,1}$s_3$}%
}}}}
\put(10310,-2180){\makebox(0,0)[lb]{\smash{{\SetFigFont{14}{16.8}{\rmdefault}{\mddefault}{\itdefault}{\color[rgb]{1,1,1}$s_4$}%
}}}}
\put(10708,-2180){\makebox(0,0)[lb]{\smash{{\SetFigFont{14}{16.8}{\rmdefault}{\mddefault}{\itdefault}{\color[rgb]{1,1,1}$s_5$}%
}}}}
\end{picture}}
\caption{Two ways to assign weights to the lozenges of a semi-hexagon with  dents on the base.  The shaded lozenges passed by the $j$-axis are weighted by $\frac{X+Y}{2}$, the ones with label $k$ are weighted by  $\frac{Xq^{k}+Yq^{-k}}{2}$.}\label{Fig:semibase}
\end{figure}

We also need the following four basic enumerations for our proofs.

We consider a semi-hexagon of side-lengths $a, b,a+b,b$. We remove $b$ up-pointing unit triangles  along the base at the positions $s_1,s_2,\dots,s_b$ as they appear from left to right. We now assign weights to the vertical lozenges of the dented semi-hexagon as in Figure \ref{Fig:semibase}(a). In particular, the vertical lozenges with center at the point $(i,j)$ are weighted by $\frac{Xq^{i}+Yq^{-i}}{2}$. All other lozenges are weighted by $1$. 
Denote by $S_{a,b}(s_1,s_2,\dots,s_b)$ the resulting weighted region. 

It is worth noticing that the  tiling number of this dented semi-hexagon was first provided by Cohn, Larsen, and Propp \cite[Proposition 2.1]{CLP}.
A weighted version of Cohn--Larsen--Propp's result can be found in \cite[pp. 374--375]{Stanley}, in terms of  the column-strict plane partitions (or reverse semi-standard Young tableaux). The following  lemma was proved implicitly in \cite{Borodin}. 

\begin{lem}\label{semilem1} For non-negative integers $a,b$ and a sequence  $(s_i)_{i=1}^{b}$ of positive integers between $1$ and $a+b$, we have
\begin{align}\label{semieq1}
\M(S_{a,b}(s_1,s_2,\dots,s_b))&=2^{-\binom{b}{2}}q^{\sum_{i=1}^{b}(b-1)(i+1/2-2s_i)}\prod_{1\leq i < j \leq b}\frac{q^{2s_j}-q^{2s_i}}{q^{2j}-q^{2i}}\prod_{i=1}^{b}\prod_{j=1}^{i-1}(q^{2(s_i+s_j-1)}X+Y).
\end{align}
\end{lem}

Next, we consider a variation of the above weighted region. We now weight the lozenges of the semi-hexagon differently as in Figure  \ref{Fig:semibase}(b). Denote by $S'_{a,b}(s_1,s_2,\dots,s_b)$ the new weighted region.  We have the following counterpart of Lemma \ref{semilem1}. 

\begin{lem}\label{semilem2} For non-negative integers $a,b$ and a sequence  $(s_i)_{i=1}^{b}$ of positive integers between $1$ and $a+b$
\begin{align}\label{semieq2}
\M(S'_{a,b}(s_1,s_2,\dots,s_b))&=2^{\sum_{i=1}^{b}(i-s_i)}q^{\sum_{i=1}^{b}\frac{(i-s_i)(s_i-3+i)}{2}}\prod_{1\leq i < j \leq b}\frac{q^{2s_j}-q^{2s_i}}{q^{2j}-q^{2i}}\prod_{i=1}^{b}\prod_{j=1}^{s_i-i}(q^{2(i+j-b-1)}X+Y).
\end{align}
\end{lem}

\begin{figure}\centering
\includegraphics[width=10cm]{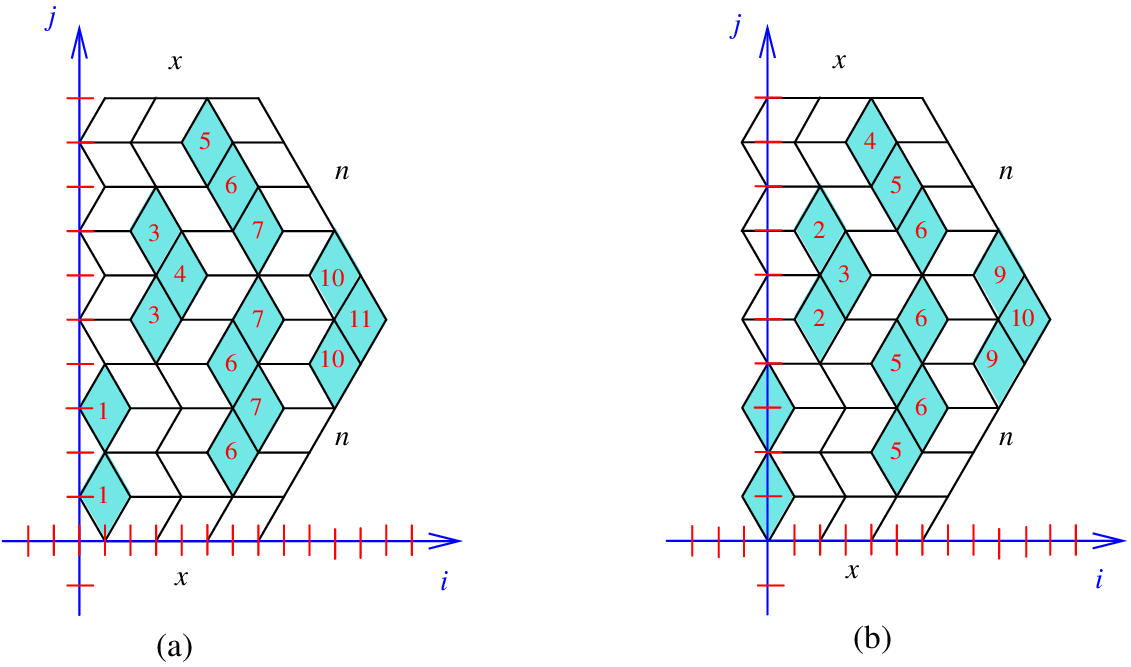}
\caption{Two ways to assign weights to lozenges of a halved hexagon. The shaded lozenges passed by the $j$-axis are weighted by $\frac{1}{2}$, the ones with label $k$ are weighted by  $\frac{q^{k}+q^{-k}}{2}$.}\label{Fig:halvedhex}
\end{figure}

Next, we consider a pentagonal region with side-lengths $x,n,n,x,2n$, whose vertical right  side runs along a zigzag path with $2n$ steps.  We now assign the weight to lozenges of the region as in Figure \ref{Fig:halvedhex}(a): the vertical lozenges with center at the point $(i,j)$ are weighted by $\frac{q^{i}+q^{-i}}{2}$; other lozenges are all weighted by $1$. Denote by $P_{x,n}$ the resulting weighted region. We usually call  $P_{n,x}$ a (weighted) \emph{halved hexagon}, as it can be viewed as half of a symmetric hexagon of side-lengths $2x+1,n,n,2x+1,n,n$ divided  along a vertical zigzag cut.

The \emph{$q$-integer} $[n]_q$ is defined as $[n]_q=1+q+\cdots+q^{n-1}$, where $[0]_q=0$. Then the \emph{$q$-factorial} is defined to be  the product of consecutive $q$-integers: $[n]_q!=[1]_q[2]_q\cdots[n]_q$, where $[0]_q!=1$.

\begin{lem}\label{proctorlem} Assume that $x,n$ are non-negative integers. Then we have
\begin{align}
\M(P_{x,n})=\frac{2^{-n^2}q^{-\sum_{i=1}^{n}(2i-1)(2x+i)}}{[1]_{q^2}![3]_{q^2}!\cdots[2n-1]_{q^2}!}\prod_{i=1}^{n}[4(x+i)]_{q^2}\prod_{1\leq i <j\leq n}[2(2x+i+j)]_{q^2}[2(j-i)]_{q^2}.
\end{align}
\end{lem}

We also consider a variant of $P_{x,n}$ obtained by re-assigning the lozenge-weights as in Figure \ref{Fig:weight2}(b). In particular, the vertical lozenges with center at the point $(i,j)$ are still weighted by $\frac{q^{i}+q^{-i}}{2}$, except for the ones intersected by the $j$-axis, which are weighted by $1/2$ (\emph{not} $1=\frac{q^{0}+q^{-0}}{2}$). Denote by $P'_{x,n}$ the new weighted region.

\begin{lem}\label{proctorlem2} Assume that $x,n$ are non-negative integers. Then we have
\begin{align}
\M(P'_{x,n})=\frac{2^{-n^2}q^{-\sum_{i=1}^{n}(2i-1)(2x+i-1)}}{[1]_{q^2}![3]_{q^2}!\cdots[2n-1]_{q^2}!}\prod_{i=1}^{n}[4(x+i)-2]_{q^2}\prod_{1\leq i <j\leq n}[2(2x+i+j-1)]_{q^2}[2(j-i)]_{q^2}.
\end{align}
\end{lem}

 Lemmas \ref{semilem2}, \ref{proctorlem}, and \ref{proctorlem2} are fairly easy to prove. However, as the author's knowledge, there are no references for these lemmas in the literature. For the completeness of the paper, we will provide the proofs of these lemmas in the Appendix. 

\section{Proofs of Main Theorems}\label{Sec:Proof} 

\begin{figure}\centering
\includegraphics[width=12cm]{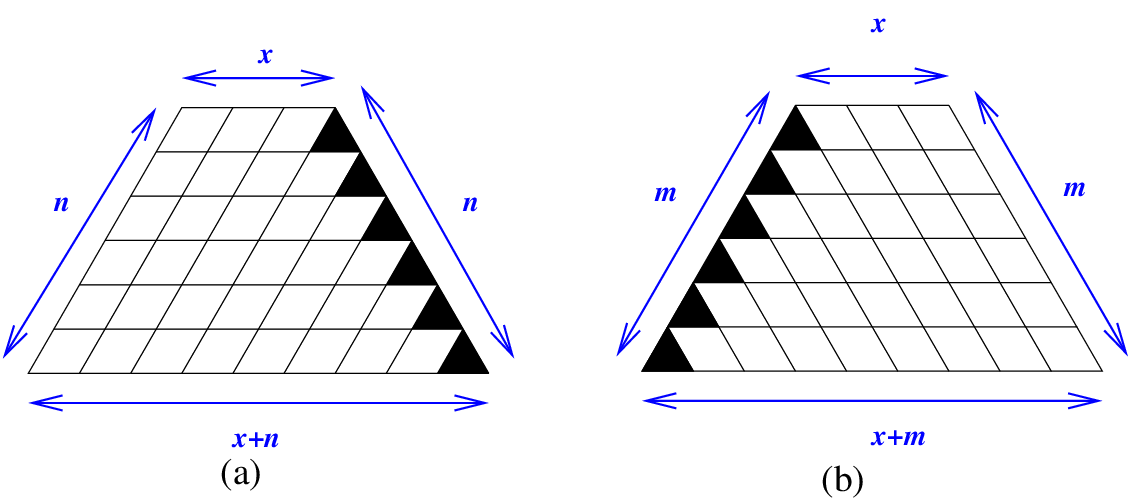}
\caption{Two base cases in the proof of Theorem \ref{semithm1}: (a) the case $m=0$ and (b) the case $n=0$.}\label{Fig:Semitwodent11}
\end{figure}

\begin{figure}\centering
\includegraphics[width=12cm]{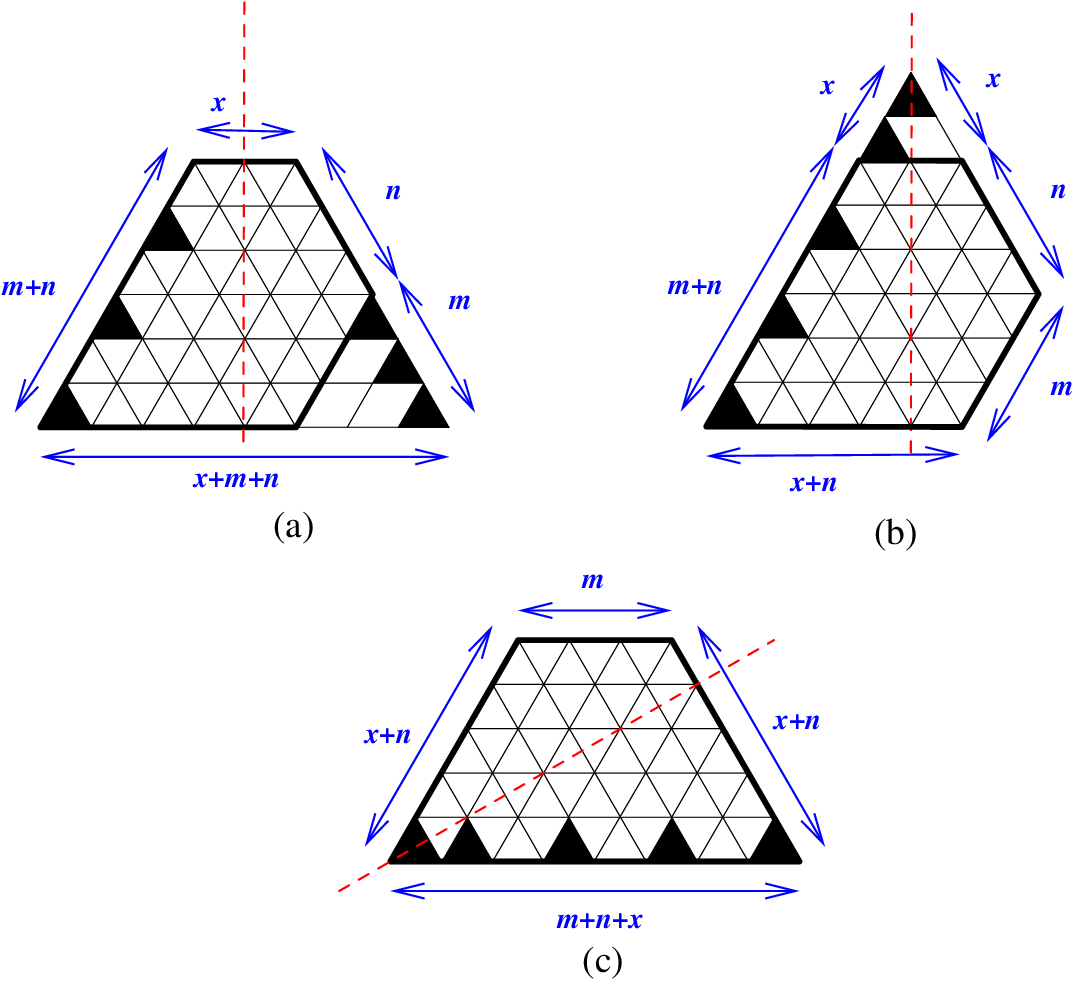}
\caption{The case $t=0$ in the proof of Theorem \ref{semithm1}.}\label{Fig:Semitwodent12}
\end{figure}

\begin{figure}\centering
\setlength{\unitlength}{3947sp}%
\begingroup\makeatletter\ifx\SetFigFont\undefined%
\gdef\SetFigFont#1#2#3#4#5{%
  \reset@font\fontsize{#1}{#2pt}%
  \fontfamily{#3}\fontseries{#4}\fontshape{#5}%
  \selectfont}%
\fi\endgroup%
\resizebox{!}{7cm}{
\begin{picture}(0,0)%
\includegraphics{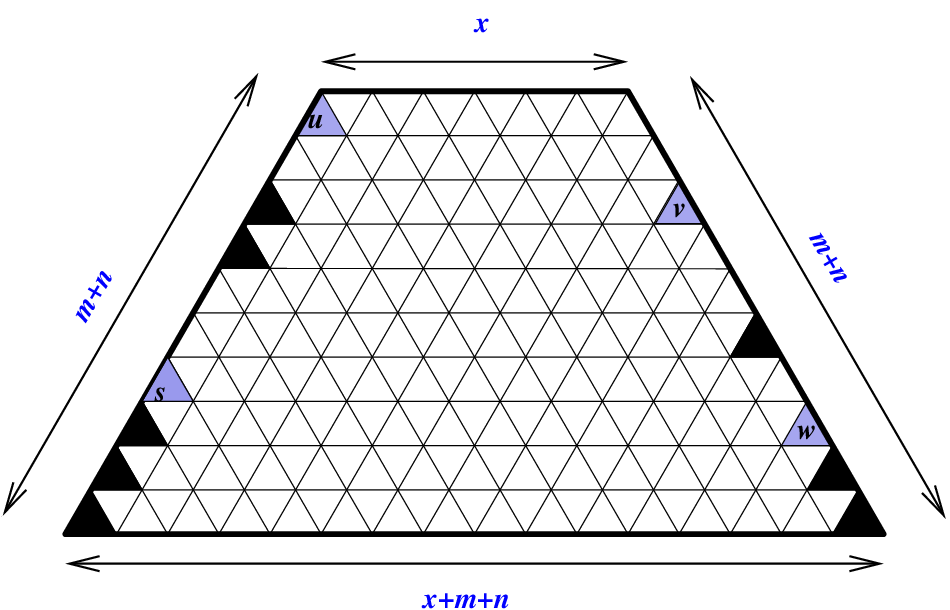}%
\end{picture}%
\begin{picture}(7588,4869)(7251,-4257)
\put(9066,-1441){\makebox(0,0)[lb]{\smash{{\SetFigFont{14}{16.8}{\rmdefault}{\bfdefault}{\updefault}{\color[rgb]{1,1,1}$a_2$}%
}}}}
\put(14011,-3581){\makebox(0,0)[lb]{\smash{{\SetFigFont{14}{16.8}{\rmdefault}{\bfdefault}{\updefault}{\color[rgb]{1,1,1}$b_4$}%
}}}}
\put(13776,-3251){\makebox(0,0)[lb]{\smash{{\SetFigFont{14}{16.8}{\rmdefault}{\bfdefault}{\updefault}{\color[rgb]{1,1,1}$b_3$}%
}}}}
\put(13191,-2161){\makebox(0,0)[lb]{\smash{{\SetFigFont{14}{16.8}{\rmdefault}{\bfdefault}{\updefault}{\color[rgb]{1,1,1}$b_2$}%
}}}}
\put(7826,-3561){\makebox(0,0)[lb]{\smash{{\SetFigFont{14}{16.8}{\rmdefault}{\bfdefault}{\updefault}{\color[rgb]{1,1,1}$a_6$}%
}}}}
\put(8036,-3201){\makebox(0,0)[lb]{\smash{{\SetFigFont{14}{16.8}{\rmdefault}{\bfdefault}{\updefault}{\color[rgb]{1,1,1}$a_5$}%
}}}}
\put(8246,-2861){\makebox(0,0)[lb]{\smash{{\SetFigFont{14}{16.8}{\rmdefault}{\bfdefault}{\updefault}{\color[rgb]{1,1,1}$a_4$}%
}}}}
\put(9286,-1091){\makebox(0,0)[lb]{\smash{{\SetFigFont{14}{16.8}{\rmdefault}{\bfdefault}{\updefault}{\color[rgb]{1,1,1}$a_1$}%
}}}}
\put(8026,-2439){\makebox(0,0)[lb]{\smash{{\SetFigFont{14}{16.8}{\rmdefault}{\bfdefault}{\updefault}{\color[rgb]{0,0,0}$a_3$}%
}}}}
\put(12871,-1036){\makebox(0,0)[lb]{\smash{{\SetFigFont{14}{16.8}{\rmdefault}{\bfdefault}{\updefault}{\color[rgb]{0,0,0}$b_1$}%
}}}}
\end{picture}}
\caption{How to apply Kuo condensation to a semi-hexagon with dents on two sides.}\label{Fig:Semitwodent3}
\end{figure}
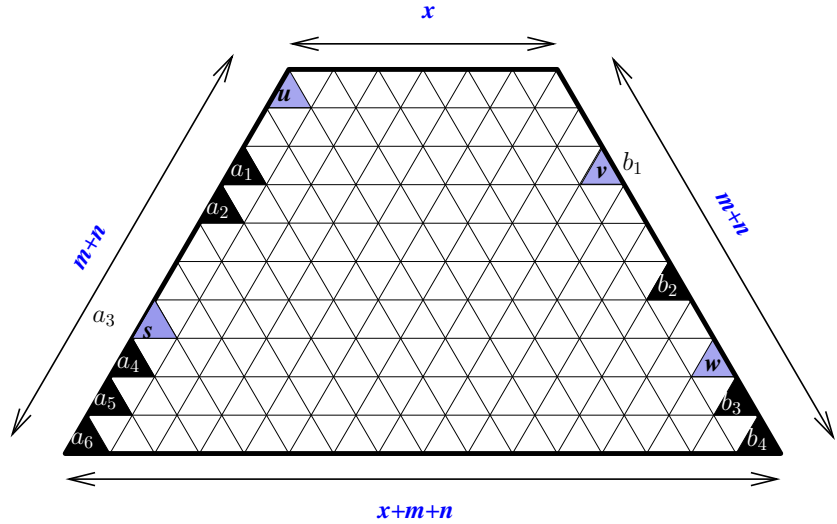

\begin{figure}\centering
\setlength{\unitlength}{3947sp}%
\begingroup\makeatletter\ifx\SetFigFont\undefined%
\gdef\SetFigFont#1#2#3#4#5{%
  \reset@font\fontsize{#1}{#2pt}%
  \fontfamily{#3}\fontseries{#4}\fontshape{#5}%
  \selectfont}%
\fi\endgroup%
\resizebox{!}{17cm}{
\begin{picture}(0,0)%
\includegraphics{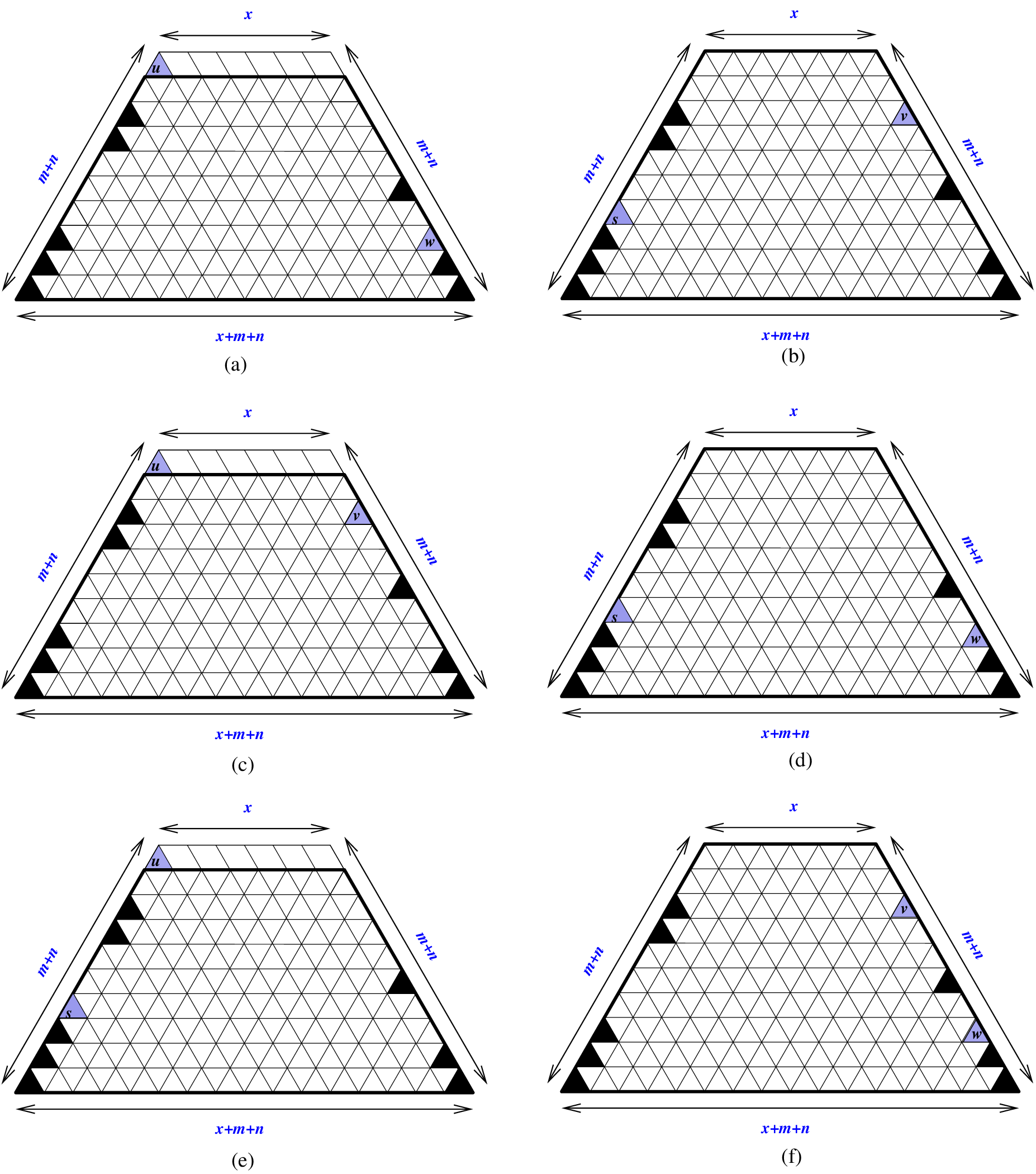}%
\end{picture}%
%
%

\begin{picture}(14821,16785)(169,-16152)
\put(657,-14539){\makebox(0,0)[lb]{\smash{{\SetFigFont{14}{16.8}{\rmdefault}{\bfdefault}{\updefault}{\color[rgb]{1,1,1}$a_5$}%
}}}}
\put(13634,-2136){\makebox(0,0)[lb]{\smash{{\SetFigFont{14}{16.8}{\rmdefault}{\bfdefault}{\updefault}{\color[rgb]{1,1,1}$b_2$}%
}}}}
\put(13627,-13484){\makebox(0,0)[lb]{\smash{{\SetFigFont{14}{16.8}{\rmdefault}{\bfdefault}{\updefault}{\color[rgb]{1,1,1}$b_2$}%
}}}}
\put(454,-3551){\makebox(0,0)[lb]{\smash{{\SetFigFont{14}{16.8}{\rmdefault}{\bfdefault}{\updefault}{\color[rgb]{1,1,1}$a_6$}%
}}}}
\put(664,-3191){\makebox(0,0)[lb]{\smash{{\SetFigFont{14}{16.8}{\rmdefault}{\bfdefault}{\updefault}{\color[rgb]{1,1,1}$a_5$}%
}}}}
\put(874,-2851){\makebox(0,0)[lb]{\smash{{\SetFigFont{14}{16.8}{\rmdefault}{\bfdefault}{\updefault}{\color[rgb]{1,1,1}$a_4$}%
}}}}
\put(1694,-1431){\makebox(0,0)[lb]{\smash{{\SetFigFont{14}{16.8}{\rmdefault}{\bfdefault}{\updefault}{\color[rgb]{1,1,1}$a_2$}%
}}}}
\put(1914,-1081){\makebox(0,0)[lb]{\smash{{\SetFigFont{14}{16.8}{\rmdefault}{\bfdefault}{\updefault}{\color[rgb]{1,1,1}$a_1$}%
}}}}
\put(6639,-3571){\makebox(0,0)[lb]{\smash{{\SetFigFont{14}{16.8}{\rmdefault}{\bfdefault}{\updefault}{\color[rgb]{1,1,1}$b_4$}%
}}}}
\put(6404,-3241){\makebox(0,0)[lb]{\smash{{\SetFigFont{14}{16.8}{\rmdefault}{\bfdefault}{\updefault}{\color[rgb]{1,1,1}$b_3$}%
}}}}
\put(5819,-2151){\makebox(0,0)[lb]{\smash{{\SetFigFont{14}{16.8}{\rmdefault}{\bfdefault}{\updefault}{\color[rgb]{1,1,1}$b_2$}%
}}}}
\put(447,-14899){\makebox(0,0)[lb]{\smash{{\SetFigFont{14}{16.8}{\rmdefault}{\bfdefault}{\updefault}{\color[rgb]{1,1,1}$a_6$}%
}}}}
\put(867,-14199){\makebox(0,0)[lb]{\smash{{\SetFigFont{14}{16.8}{\rmdefault}{\bfdefault}{\updefault}{\color[rgb]{1,1,1}$a_4$}%
}}}}
\put(8262,-9229){\makebox(0,0)[lb]{\smash{{\SetFigFont{14}{16.8}{\rmdefault}{\bfdefault}{\updefault}{\color[rgb]{1,1,1}$a_6$}%
}}}}
\put(8472,-8869){\makebox(0,0)[lb]{\smash{{\SetFigFont{14}{16.8}{\rmdefault}{\bfdefault}{\updefault}{\color[rgb]{1,1,1}$a_5$}%
}}}}
\put(8682,-8529){\makebox(0,0)[lb]{\smash{{\SetFigFont{14}{16.8}{\rmdefault}{\bfdefault}{\updefault}{\color[rgb]{1,1,1}$a_4$}%
}}}}
\put(9502,-7109){\makebox(0,0)[lb]{\smash{{\SetFigFont{14}{16.8}{\rmdefault}{\bfdefault}{\updefault}{\color[rgb]{1,1,1}$a_2$}%
}}}}
\put(9722,-6759){\makebox(0,0)[lb]{\smash{{\SetFigFont{14}{16.8}{\rmdefault}{\bfdefault}{\updefault}{\color[rgb]{1,1,1}$a_1$}%
}}}}
\put(14447,-9249){\makebox(0,0)[lb]{\smash{{\SetFigFont{14}{16.8}{\rmdefault}{\bfdefault}{\updefault}{\color[rgb]{1,1,1}$b_4$}%
}}}}
\put(14212,-8919){\makebox(0,0)[lb]{\smash{{\SetFigFont{14}{16.8}{\rmdefault}{\bfdefault}{\updefault}{\color[rgb]{1,1,1}$b_3$}%
}}}}
\put(13627,-7829){\makebox(0,0)[lb]{\smash{{\SetFigFont{14}{16.8}{\rmdefault}{\bfdefault}{\updefault}{\color[rgb]{1,1,1}$b_2$}%
}}}}
\put(1687,-12779){\makebox(0,0)[lb]{\smash{{\SetFigFont{14}{16.8}{\rmdefault}{\bfdefault}{\updefault}{\color[rgb]{1,1,1}$a_2$}%
}}}}
\put(1907,-12429){\makebox(0,0)[lb]{\smash{{\SetFigFont{14}{16.8}{\rmdefault}{\bfdefault}{\updefault}{\color[rgb]{1,1,1}$a_1$}%
}}}}
\put(6632,-14919){\makebox(0,0)[lb]{\smash{{\SetFigFont{14}{16.8}{\rmdefault}{\bfdefault}{\updefault}{\color[rgb]{1,1,1}$b_4$}%
}}}}
\put(6397,-14589){\makebox(0,0)[lb]{\smash{{\SetFigFont{14}{16.8}{\rmdefault}{\bfdefault}{\updefault}{\color[rgb]{1,1,1}$b_3$}%
}}}}
\put(447,-9244){\makebox(0,0)[lb]{\smash{{\SetFigFont{14}{16.8}{\rmdefault}{\bfdefault}{\updefault}{\color[rgb]{1,1,1}$a_6$}%
}}}}
\put(657,-8884){\makebox(0,0)[lb]{\smash{{\SetFigFont{14}{16.8}{\rmdefault}{\bfdefault}{\updefault}{\color[rgb]{1,1,1}$a_5$}%
}}}}
\put(867,-8544){\makebox(0,0)[lb]{\smash{{\SetFigFont{14}{16.8}{\rmdefault}{\bfdefault}{\updefault}{\color[rgb]{1,1,1}$a_4$}%
}}}}
\put(1687,-7124){\makebox(0,0)[lb]{\smash{{\SetFigFont{14}{16.8}{\rmdefault}{\bfdefault}{\updefault}{\color[rgb]{1,1,1}$a_2$}%
}}}}
\put(1907,-6774){\makebox(0,0)[lb]{\smash{{\SetFigFont{14}{16.8}{\rmdefault}{\bfdefault}{\updefault}{\color[rgb]{1,1,1}$a_1$}%
}}}}
\put(6632,-9264){\makebox(0,0)[lb]{\smash{{\SetFigFont{14}{16.8}{\rmdefault}{\bfdefault}{\updefault}{\color[rgb]{1,1,1}$b_4$}%
}}}}
\put(6397,-8934){\makebox(0,0)[lb]{\smash{{\SetFigFont{14}{16.8}{\rmdefault}{\bfdefault}{\updefault}{\color[rgb]{1,1,1}$b_3$}%
}}}}
\put(5812,-7844){\makebox(0,0)[lb]{\smash{{\SetFigFont{14}{16.8}{\rmdefault}{\bfdefault}{\updefault}{\color[rgb]{1,1,1}$b_2$}%
}}}}
\put(5812,-13499){\makebox(0,0)[lb]{\smash{{\SetFigFont{14}{16.8}{\rmdefault}{\bfdefault}{\updefault}{\color[rgb]{1,1,1}$b_2$}%
}}}}
\put(8262,-14884){\makebox(0,0)[lb]{\smash{{\SetFigFont{14}{16.8}{\rmdefault}{\bfdefault}{\updefault}{\color[rgb]{1,1,1}$a_6$}%
}}}}
\put(8472,-14524){\makebox(0,0)[lb]{\smash{{\SetFigFont{14}{16.8}{\rmdefault}{\bfdefault}{\updefault}{\color[rgb]{1,1,1}$a_5$}%
}}}}
\put(8682,-14184){\makebox(0,0)[lb]{\smash{{\SetFigFont{14}{16.8}{\rmdefault}{\bfdefault}{\updefault}{\color[rgb]{1,1,1}$a_4$}%
}}}}
\put(9502,-12764){\makebox(0,0)[lb]{\smash{{\SetFigFont{14}{16.8}{\rmdefault}{\bfdefault}{\updefault}{\color[rgb]{1,1,1}$a_2$}%
}}}}
\put(9722,-12414){\makebox(0,0)[lb]{\smash{{\SetFigFont{14}{16.8}{\rmdefault}{\bfdefault}{\updefault}{\color[rgb]{1,1,1}$a_1$}%
}}}}
\put(14447,-14904){\makebox(0,0)[lb]{\smash{{\SetFigFont{14}{16.8}{\rmdefault}{\bfdefault}{\updefault}{\color[rgb]{1,1,1}$b_4$}%
}}}}
\put(14212,-14574){\makebox(0,0)[lb]{\smash{{\SetFigFont{14}{16.8}{\rmdefault}{\bfdefault}{\updefault}{\color[rgb]{1,1,1}$b_3$}%
}}}}
\put(8269,-3536){\makebox(0,0)[lb]{\smash{{\SetFigFont{14}{16.8}{\rmdefault}{\bfdefault}{\updefault}{\color[rgb]{1,1,1}$a_6$}%
}}}}
\put(8479,-3176){\makebox(0,0)[lb]{\smash{{\SetFigFont{14}{16.8}{\rmdefault}{\bfdefault}{\updefault}{\color[rgb]{1,1,1}$a_5$}%
}}}}
\put(8689,-2836){\makebox(0,0)[lb]{\smash{{\SetFigFont{14}{16.8}{\rmdefault}{\bfdefault}{\updefault}{\color[rgb]{1,1,1}$a_4$}%
}}}}
\put(9509,-1416){\makebox(0,0)[lb]{\smash{{\SetFigFont{14}{16.8}{\rmdefault}{\bfdefault}{\updefault}{\color[rgb]{1,1,1}$a_2$}%
}}}}
\put(9729,-1066){\makebox(0,0)[lb]{\smash{{\SetFigFont{14}{16.8}{\rmdefault}{\bfdefault}{\updefault}{\color[rgb]{1,1,1}$a_1$}%
}}}}
\put(14454,-3556){\makebox(0,0)[lb]{\smash{{\SetFigFont{14}{16.8}{\rmdefault}{\bfdefault}{\updefault}{\color[rgb]{1,1,1}$b_4$}%
}}}}
\put(14219,-3226){\makebox(0,0)[lb]{\smash{{\SetFigFont{14}{16.8}{\rmdefault}{\bfdefault}{\updefault}{\color[rgb]{1,1,1}$b_3$}%
}}}}
\end{picture}}
\caption{Obtaining a recurrence for tiling generating functions of  the semi-hexagons with dents on two sides.}\label{Fig:Semitwodent4}
\end{figure}

\begin{proof}[Proof of Theorem \ref{semithm1}]
We rewrite identity (\ref{maineq1}) as
\begin{equation}\label{semieq1b}
\M(S_{x}((a_i)_{i=1}^{m}; (b_j)_{j=1}^{n}))=f_{x,y}((a_i)_{i=1}^{m}; (b_j)_{j=1}^{n})\cdot \M(S_{y}((a_i)_{i=1}^{m}; (b_j)_{j=1}^{n})),
\end{equation}
where $f_{x,y}((a_i)_{i=1}^{m}; (b_j)_{j=1}^{n})$ denotes the expression on the right-hand side of (\ref{maineq1}), i.e.,
\begin{align}
f_{x,y}((a_i)_{i=1}^{m}; (b_j)_{j=1}^{n})=&q^{(y-x)\left(\sum_{i=1}^{m}a_i+\sum_{j=1}^{n}b_j-\frac{(m+n)(m+n+1)}{2}\right)}\frac{\PP_{q^2}(y,m,n)}{\PP_{q^2}(x,m,n)}\notag\\
&\times\prod_{i=1}^{m}\frac{(q^{2(x+i)};q^2)_{a_i-i}}{(q^{2(y+i)};q^2)_{a_i-i}}\prod_{j=1}^{n}\frac{(q^{2(x+j)};q^2)_{b_j-j}}{(q^{2(y+j)};q^2)_{b_j-j}}.
\end{align}

We prove (\ref{semieq1b}) by induction on the statistic $p:=m+2n+t$, where $t=\sum_{i=1}^{n}((m+i)-b_{i})$. The parameter $t$ roughly measures how close the $b$-dents to the base of the semi-hexagon. When $t=0$, all $b$-dents are clustering to the lower-right corner of the region (see Figure \ref{Fig:Semitwodent12}(a)). The base cases are the situations when at least one of the perimeters $m,n,t$ is equal to $0$.

If $m=0$, then our two semi-hexagons $S_{x}(\textbf{a},\textbf{b})$ and  $S_{y}(\textbf{a},\textbf{b})$ have exactly 1 tiling as shown in Figure \ref{Fig:Semitwodent11}(a). In this case, identity (\ref{semieq1b}) becomes ``1=1." The case $n=0$ is similar (illustrated in Figure \ref{Fig:Semitwodent11}(b)). 

If $t=0$, then all $b$-dents of $S_{x}((a_i)_{i=1}^{m}; (b_j)_{j=1}^{n})$ are clustering to the lower-right corner. By removing forced lozenges, we get a pentagonal region with dents on the left side (illustrated by the region restricted by the bold contour in Figure \ref{Fig:Semitwodent12}(a)). The resulting region has the same tiling generating function as the region in Figure \ref{Fig:Semitwodent12}(b). (The two regions differ by several forced lozenges with weight 1.) We now $120^{\circ}$-rotate this region to get a weighted semi-hexagon in Lemma \ref{semilem1} (see Figure \ref{Fig:Semitwodent12}(c)).  (The dotted lines indicate the lines containing the $j$-axes of the regions. The lozenges intersecting to these dotted lines  are weighted by $\frac{X+Y}{2}$.) This way, we obtain an explicit formula for the tiling generating function of the semi-hexagon $S_{x}((a_i)_{i=1}^{m}; (b_j)_{j=1}^{n})$ when $t=0$. Working similarly, we get a formula for the tiling generating function of $S_{y}((a_i)_{i=1}^{m}; (b_j)_{j=1}^{n})$. Identity (\ref{semieq1b}) follows directly from Lemma \ref{semilem1} in this case.

\medskip

For the induction step, we assume that $m,n,t$ are all positive and that (\ref{semieq1b}) holds for any pair of  semi-hexagons whose $p$-statistic is strictly less than $m+2n+t$. 

By the tile-ability of the semi-hexagons in Lemma \ref{tileability1}, at least one of $a_1$ and $b_1$ is strictly greater than 1. Without loss of generality, we assume that $a_1>1$. If $b_1=1$, then one can remove forced lozenges on the top of the two semi-hexagons to get two ``smaller"\footnote{In the rest of this proof, we say that the semi-hexagon $A$ is \emph{smaller} than the semi-hexagon $B$, if the $p$-statistic of $A$ is less than that of $B$.} semi-hexagons of the same type. Then  (\ref{semieq1b}) follows from the induction hypothesis. Therefore, we can also assume that $b_1>1$.

When $a_1,b_1>1$, we will show that the expressions on both sides of (\ref{semieq1b}) satisfy the same recurrence. Then the theorem follows from the induction principle. 

To obtain the recurrence for the tiling generating function of $S=S_{x}((a_i)_{i=1}^{m}; (b_j)_{j=1}^{n})$ on the left-hand side of (\ref{semieq1b}), we use Kuo's condensation in  Lemma \ref{kuothm2}. Assume that $l$ is the largest index such that there is no $a$-dent at the position $a_l-1$ on the left side of $S$.   We consider the region $R$ obtained by filling the dents at the positions of $a_l$ and $b_1$ in $S$ by two unit triangles  (see Figure \ref{Fig:Semitwodent3} for an example when $l=3$). $R$ now has two more up-pointing unit triangles than down-pointing unit triangles. We apply Kuo condensation in Lemma \ref{kuothm2} to the dual graph $G$ of  $R$ with the four vertices $u,v,w,s$ corresponding to the shaded up-pointing unit triangles of the same label. More precisely, the $u$-triangle is the up-pointing triangle on the upper-left corner of $S$, the $v$-triangle is at the position $b_1$, the $w$-triangle is at the last non-dent position on the right side of $S$, and the $s$-triangle is at the position $a_l$. Let $\alpha$ denote the position of the $w$-triangle. We get a recurrence:
\begin{equation}\label{kuoeq2b}
\MM(G-\{u,w\})\MM(G-\{v,s\})=\MM(G-\{u,v\})\MM(G-\{w,s\})+\MM(G-\{u,s\})\MM(G-\{v,w\}).
\end{equation}
 We plan to convert each matching generating function in the recurrence into the tiling generating function of a semi-hexagon. 
 
For brevity,  we use the notations $\textbf{a}-c$  and $\textbf{a}+d$ for the sequences obtained from $\textbf{a}=(a_i)_{i=1}^{m}$ by excluding the term $c$ and by including the term $d$ (and rearranging in increasing order), respectively. We also use the notation $\textbf{a}^{*}$ for the sequence obtained from the sequence $\textbf{a}$ by subtracting 1 from each of its term. 

First, we consider the region corresponding to graph $G-\{u,w\}$ (as shown in Figure \ref{Fig:Semitwodent4}(a)). The removal of the $u$- and $w$-triangles yields forced lozenges (with weight $1$) on the top of the region. Removal of these forced lozenges gives a new semi-hexagon, namely $S_{x+1}((\textbf{a}-a_l)^{*}; (\textbf{b}-b_1+\alpha)^{*})$. In this semi-hexagon, the sequence $(\textbf{a}-a_l)^{*}$ is obtained from the sequence $\textbf{a}$ by excluding the term $a_l$ and then subtracting $1$ from each term of the resulting sequence, similarly, the sequence $(\textbf{b}-b_1+\alpha)^{*}$ is obtained from the sequence $\textbf{b}$ by excluding the term $b_1$, including the term $\alpha$, and then subtracting $1$ from each term of the resulting sequence. We get 
\begin{equation}
\MM(G-\{u,w\})=\M(S_{x+1}((\textbf{a}-a_l)^{*}; (\textbf{b}-b_1+\alpha)^{*})).
\end{equation}

Working similarly for the regions corresponding to the other five graphs in recurrence (\ref{kuoeq2b}), based on Figures \ref{Fig:Semitwodent4}(b)--(f), we get 
\begin{equation}
\MM(G-\{v,s\})=\M(S_{x}(\textbf{a}; \textbf{b})),
\end{equation}
\begin{equation}
\MM(G-\{u,v\})=\M(S_{x+1}((\textbf{a}-a_l)^{*}; \textbf{b}^{*})),
\end{equation}
\begin{equation}
\MM(G-\{w,s\})=\M(S_{x}(\textbf{a}; \textbf{b}-b_1+\alpha)),
\end{equation}
\begin{equation}
\MM(G-\{u,s\})=\M(S_{x+1}(\textbf{a}^{*}; (\textbf{b}-b_1)^{*})),
\end{equation}
\begin{equation}
\MM(G-\{v,s\})=\M(S_{x}(\textbf{a}-a_l; \textbf{b}+\alpha)).
\end{equation}
These six equations transform recurrence (\ref{kuoeq2b}) into a the recurrence for the tiling generating functions of the semi-hexagons:
\begin{align}\label{recur1}
\M(S_{x+1}((\textbf{a}-a_l)^{*}; (\textbf{b}-b_1+\alpha)^{*})\M(S_{x}(\textbf{a}; \textbf{b}))&= \M(S_{x+1}((\textbf{a}-a_l)^{*}; \textbf{b}^{*}))\M(S_{x}(\textbf{a}; \textbf{b}-b_1+\alpha))\notag\\
&+\M(S_{x+1}(\textbf{a}^{*}; (\textbf{b}-b_1)^{*}))\M(S_{x}(\textbf{a}-a_l; \textbf{b}+\alpha)).
\end{align}

We note that if $S=S_{x}(\textbf{a}; \textbf{b})$ is tile-able, then all other five semi-hexagons in the above recurrence are also tile-able (by Lemma \ref{tileability1}).  Moreover, one could verify that the $p$-statistics of these five semi-hexagons are all strictly less than $m+2n+t$.

\medskip

To complete the prove, we want to show that the expression $f_{x,y}(\textbf{a};\textbf{b}) \cdot \M(S_{y}(\textbf{a};\textbf{b}))$ on the right-hand side of (\ref{semieq1b}) also satisfies the same recurrence. Equivalently, we need to verify that
\begin{align}
A\cdot \M(S_{y+1}((\textbf{a}-a_l)^{*}; (\textbf{b}-b_1+\alpha)^{*})\M(S_{y}(\textbf{a}; \textbf{b}))&= B\cdot \M(S_{y+1}((\textbf{a}-a_l)^{*}; \textbf{b}^{*}))\M(S_{y}(\textbf{a}; \textbf{b}-b_1+\alpha))\notag\\
&+C\cdot \M(S_{y+1}(\textbf{a}^{*}; (\textbf{b}-b_1)^{*}))\M(S_{y}(\textbf{a}-a_l; \textbf{b}+\alpha)),
\end{align}
where 
\[A=f_{x+1,y+1}((\textbf{a}-a_l)^{*}; (\textbf{b}-b_1+\alpha)^{*})\cdot f_{x,y}(\textbf{a}; \textbf{b})\]
\[B=f_{x+1,y+1}((\textbf{a}-a_l)^{*}; \textbf{b}^{*})\cdot f_{x,y}(\textbf{a}; \textbf{b}-b_1+\alpha)\]
\[C=f_{x+1,y+1}(\textbf{a}^{*}; (\textbf{b}-b_1)^{*})\cdot f_{x,y}(\textbf{a}-a_l; \textbf{b}+\alpha).\]
It is routine to verify that $ A=B=C$. This means that we now only need to verify that 


\begin{align}
 \M(S_{y+1}((\textbf{a}-a_l)^{*}; (\textbf{b}-b_1+\alpha)^{*})\M(S_{y}(\textbf{a}; \textbf{b}))&= \M(S_{y+1}((\textbf{a}-a_l)^{*}; \textbf{b}^{*}))\M(S_{y}(\textbf{a}; \textbf{b}-b_1+\alpha))\notag\\
&+ \M(S_{y+1}(\textbf{a}^{*}; (\textbf{b}-b_1)^{*}))\M(S_{y}(\textbf{a}-a_l; \textbf{b}+\alpha)).
\end{align}
 However, this recurrence follows directly from recurrence (\ref{recur1}) by simply replacing $x$ by $y$. This finishes our proof.
\end{proof}

 One could prove Theorem \ref{semithm2} in the same way as Theorem \ref{semithm1}, using Lemma \ref{semilem2}. We leave this proof as an exercise to the reader.

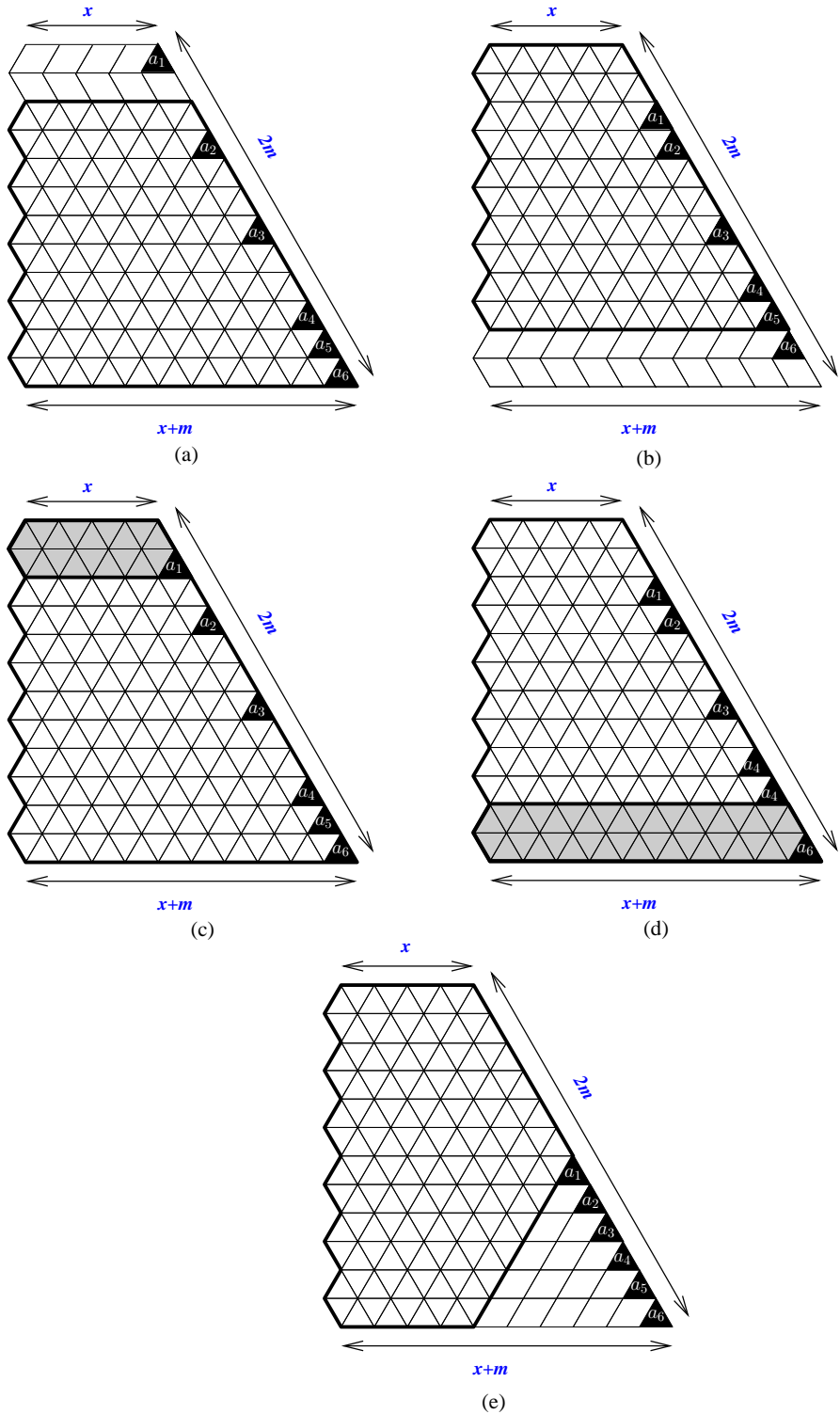
\begin{figure}\centering
\setlength{\unitlength}{3947sp}%
\begingroup\makeatletter\ifx\SetFigFont\undefined%
\gdef\SetFigFont#1#2#3#4#5{%
  \reset@font\fontsize{#1}{#2pt}%
  \fontfamily{#3}\fontseries{#4}\fontshape{#5}%
  \selectfont}%
\fi\endgroup%
\resizebox{!}{20cm}{
\begin{picture}(0,0)%
\includegraphics{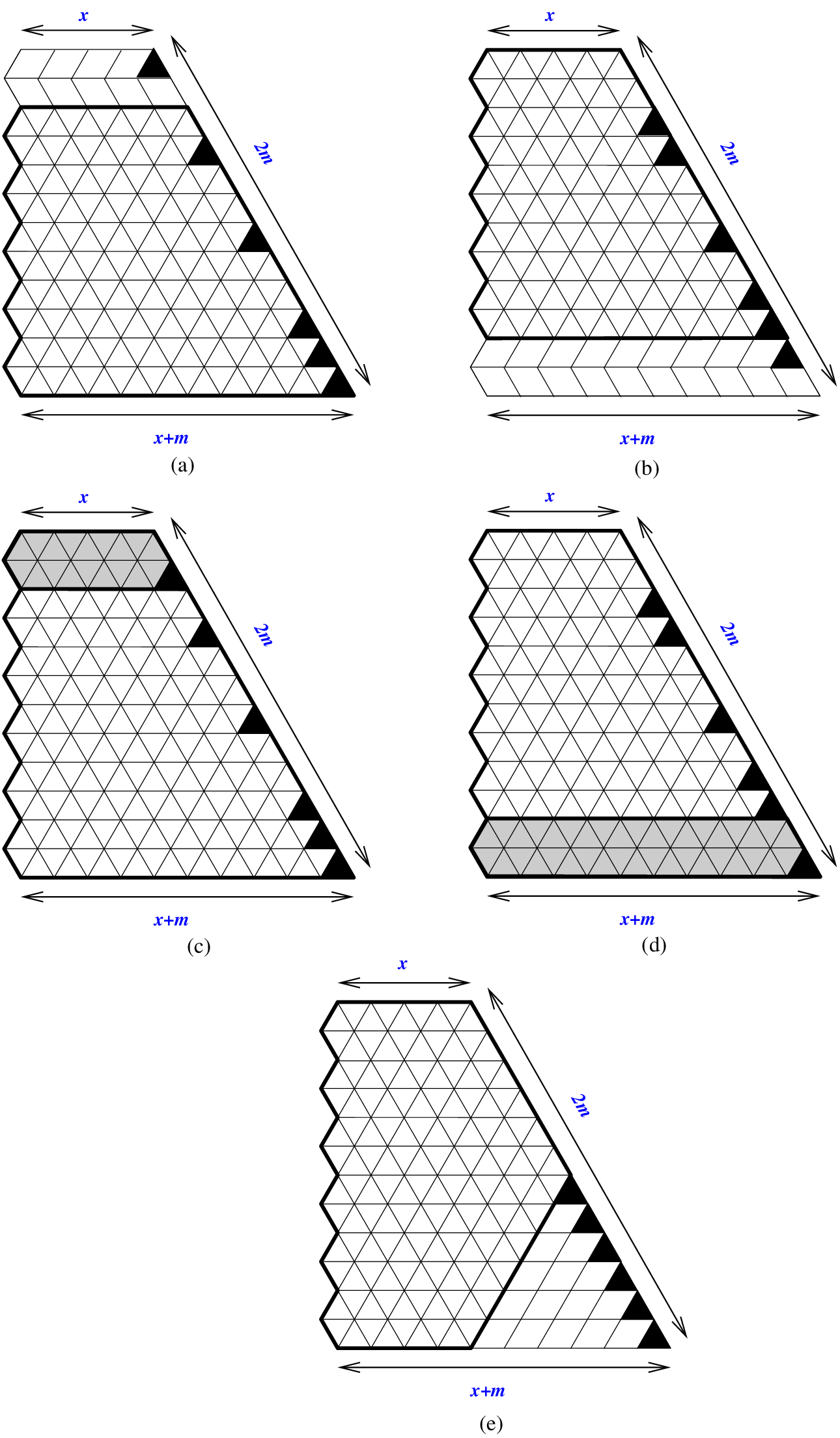}%
\end{picture}%
%
%

\begin{picture}(10303,17652)(968,-17385)
\put(10301,-9702){\makebox(0,0)[lb]{\smash{{\SetFigFont{14}{16.8}{\rmdefault}{\bfdefault}{\updefault}{\color[rgb]{1,1,1}$a_4$}%
}}}}
\put(8286,-15129){\makebox(0,0)[lb]{\smash{{\SetFigFont{14}{16.8}{\rmdefault}{\bfdefault}{\updefault}{\color[rgb]{1,1,1}$a_3$}%
}}}}
\put(8092,-14767){\makebox(0,0)[lb]{\smash{{\SetFigFont{14}{16.8}{\rmdefault}{\bfdefault}{\updefault}{\color[rgb]{1,1,1}$a_2$}%
}}}}
\put(2743,-570){\makebox(0,0)[lb]{\smash{{\SetFigFont{14}{16.8}{\rmdefault}{\bfdefault}{\updefault}{\color[rgb]{1,1,1}$a_1$}%
}}}}
\put(3375,-1654){\makebox(0,0)[lb]{\smash{{\SetFigFont{14}{16.8}{\rmdefault}{\bfdefault}{\updefault}{\color[rgb]{1,1,1}$a_2$}%
}}}}
\put(3975,-2719){\makebox(0,0)[lb]{\smash{{\SetFigFont{14}{16.8}{\rmdefault}{\bfdefault}{\updefault}{\color[rgb]{1,1,1}$a_3$}%
}}}}
\put(4590,-3769){\makebox(0,0)[lb]{\smash{{\SetFigFont{14}{16.8}{\rmdefault}{\bfdefault}{\updefault}{\color[rgb]{1,1,1}$a_4$}%
}}}}
\put(4800,-4121){\makebox(0,0)[lb]{\smash{{\SetFigFont{14}{16.8}{\rmdefault}{\bfdefault}{\updefault}{\color[rgb]{1,1,1}$a_5$}%
}}}}
\put(5010,-4466){\makebox(0,0)[lb]{\smash{{\SetFigFont{14}{16.8}{\rmdefault}{\bfdefault}{\updefault}{\color[rgb]{1,1,1}$a_6$}%
}}}}
\put(7877,-14413){\makebox(0,0)[lb]{\smash{{\SetFigFont{14}{16.8}{\rmdefault}{\bfdefault}{\updefault}{\color[rgb]{1,1,1}$a_1$}%
}}}}
\put(8482,-15470){\makebox(0,0)[lb]{\smash{{\SetFigFont{14}{16.8}{\rmdefault}{\bfdefault}{\updefault}{\color[rgb]{1,1,1}$a_4$}%
}}}}
\put(8692,-15822){\makebox(0,0)[lb]{\smash{{\SetFigFont{14}{16.8}{\rmdefault}{\bfdefault}{\updefault}{\color[rgb]{1,1,1}$a_5$}%
}}}}
\put(8901,-1281){\makebox(0,0)[lb]{\smash{{\SetFigFont{14}{16.8}{\rmdefault}{\bfdefault}{\updefault}{\color[rgb]{1,1,1}$a_1$}%
}}}}
\put(2954,-6862){\makebox(0,0)[lb]{\smash{{\SetFigFont{14}{16.8}{\rmdefault}{\bfdefault}{\updefault}{\color[rgb]{1,1,1}$a_1$}%
}}}}
\put(3374,-7575){\makebox(0,0)[lb]{\smash{{\SetFigFont{14}{16.8}{\rmdefault}{\bfdefault}{\updefault}{\color[rgb]{1,1,1}$a_2$}%
}}}}
\put(3974,-8640){\makebox(0,0)[lb]{\smash{{\SetFigFont{14}{16.8}{\rmdefault}{\bfdefault}{\updefault}{\color[rgb]{1,1,1}$a_3$}%
}}}}
\put(4589,-9690){\makebox(0,0)[lb]{\smash{{\SetFigFont{14}{16.8}{\rmdefault}{\bfdefault}{\updefault}{\color[rgb]{1,1,1}$a_4$}%
}}}}
\put(4799,-10042){\makebox(0,0)[lb]{\smash{{\SetFigFont{14}{16.8}{\rmdefault}{\bfdefault}{\updefault}{\color[rgb]{1,1,1}$a_5$}%
}}}}
\put(5009,-10387){\makebox(0,0)[lb]{\smash{{\SetFigFont{14}{16.8}{\rmdefault}{\bfdefault}{\updefault}{\color[rgb]{1,1,1}$a_6$}%
}}}}
\put(9102,-1659){\makebox(0,0)[lb]{\smash{{\SetFigFont{14}{16.8}{\rmdefault}{\bfdefault}{\updefault}{\color[rgb]{1,1,1}$a_2$}%
}}}}
\put(9702,-2724){\makebox(0,0)[lb]{\smash{{\SetFigFont{14}{16.8}{\rmdefault}{\bfdefault}{\updefault}{\color[rgb]{1,1,1}$a_3$}%
}}}}
\put(10120,-3416){\makebox(0,0)[lb]{\smash{{\SetFigFont{14}{16.8}{\rmdefault}{\bfdefault}{\updefault}{\color[rgb]{1,1,1}$a_4$}%
}}}}
\put(10330,-3768){\makebox(0,0)[lb]{\smash{{\SetFigFont{14}{16.8}{\rmdefault}{\bfdefault}{\updefault}{\color[rgb]{1,1,1}$a_5$}%
}}}}
\put(10537,-4127){\makebox(0,0)[lb]{\smash{{\SetFigFont{14}{16.8}{\rmdefault}{\bfdefault}{\updefault}{\color[rgb]{1,1,1}$a_6$}%
}}}}
\put(8881,-7201){\makebox(0,0)[lb]{\smash{{\SetFigFont{14}{16.8}{\rmdefault}{\bfdefault}{\updefault}{\color[rgb]{1,1,1}$a_1$}%
}}}}
\put(9103,-7565){\makebox(0,0)[lb]{\smash{{\SetFigFont{14}{16.8}{\rmdefault}{\bfdefault}{\updefault}{\color[rgb]{1,1,1}$a_2$}%
}}}}
\put(9703,-8630){\makebox(0,0)[lb]{\smash{{\SetFigFont{14}{16.8}{\rmdefault}{\bfdefault}{\updefault}{\color[rgb]{1,1,1}$a_3$}%
}}}}
\put(10738,-10377){\makebox(0,0)[lb]{\smash{{\SetFigFont{14}{16.8}{\rmdefault}{\bfdefault}{\updefault}{\color[rgb]{1,1,1}$a_6$}%
}}}}
\put(10081,-9352){\makebox(0,0)[lb]{\smash{{\SetFigFont{14}{16.8}{\rmdefault}{\bfdefault}{\updefault}{\color[rgb]{1,1,1}$a_4$}%
}}}}
\put(8902,-16167){\makebox(0,0)[lb]{\smash{{\SetFigFont{14}{16.8}{\rmdefault}{\bfdefault}{\updefault}{\color[rgb]{1,1,1}$a_6$}%
}}}}
\end{picture}}
\caption{Several special cases in the proof of Theorem \ref{halfthm1}: (a) the case $a_1=1$, (b) the case $t=0$, (c) the case $a_1=2$, (d) the case $t=1$ (i.e., $a_{m}=2m$ and $a_{m-1}\leq 2m-2$), and  (e) the case $t=m$.}\label{Fig:Semitwodent7}
\end{figure}

\begin{figure}\centering
\setlength{\unitlength}{3947sp}%
\begingroup\makeatletter\ifx\SetFigFont\undefined%
\gdef\SetFigFont#1#2#3#4#5{%
  \reset@font\fontsize{#1}{#2pt}%
  \fontfamily{#3}\fontseries{#4}\fontshape{#5}%
  \selectfont}%
\fi\endgroup%
\resizebox{!}{7cm}{
\begin{picture}(0,0)%
\includegraphics{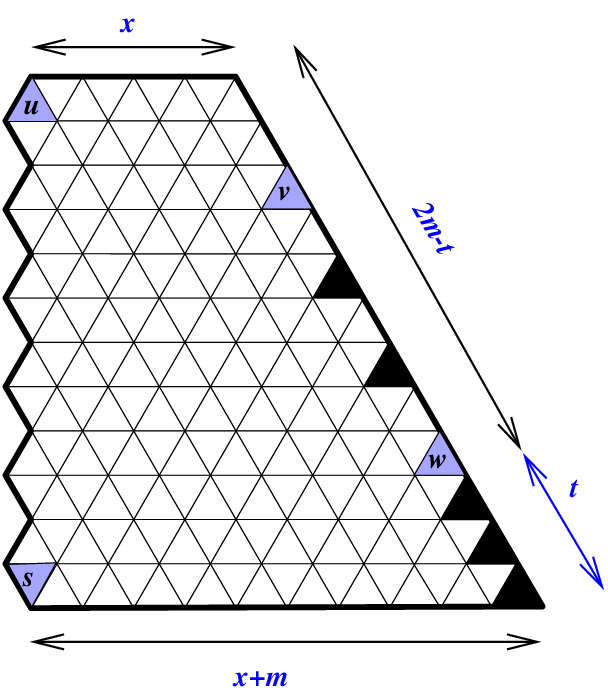}%
\end{picture}%
%
%

\begin{picture}(4855,5493)(6104,-4861)
\put(9721,-3411){\makebox(0,0)[lb]{\smash{{\SetFigFont{14}{16.8}{\rmdefault}{\bfdefault}{\updefault}{\color[rgb]{1,1,1}$a_4$}%
}}}}
\put(8604,-864){\makebox(0,0)[lb]{\smash{{\SetFigFont{14}{16.8}{\rmdefault}{\bfdefault}{\updefault}{\color[rgb]{0,0,0}$a_1$}%
}}}}
\put(9091,-2391){\makebox(0,0)[lb]{\smash{{\SetFigFont{14}{16.8}{\rmdefault}{\bfdefault}{\updefault}{\color[rgb]{1,1,1}$a_3$}%
}}}}
\put(10141,-4108){\makebox(0,0)[lb]{\smash{{\SetFigFont{14}{16.8}{\rmdefault}{\bfdefault}{\updefault}{\color[rgb]{1,1,1}$a_6$}%
}}}}
\put(9931,-3763){\makebox(0,0)[lb]{\smash{{\SetFigFont{14}{16.8}{\rmdefault}{\bfdefault}{\updefault}{\color[rgb]{1,1,1}$a_5$}%
}}}}
\put(8691,-1661){\makebox(0,0)[lb]{\smash{{\SetFigFont{14}{16.8}{\rmdefault}{\bfdefault}{\updefault}{\color[rgb]{1,1,1}$a_2$}%
}}}}
\end{picture}}
\caption{How to apply Kuo condensation to the quartered hexagon.}\label{Fig:Semitwodent8}
\end{figure}
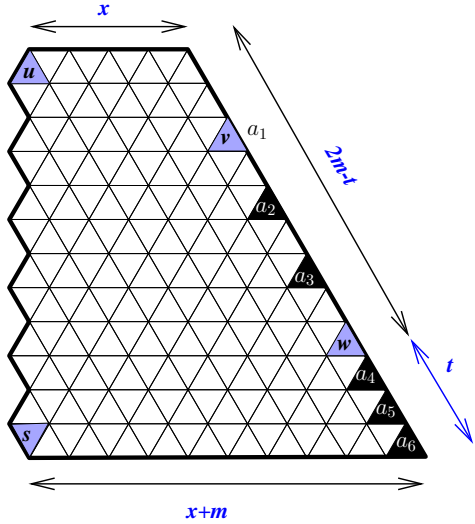

\begin{figure}\centering
\setlength{\unitlength}{3947sp}%
\begingroup\makeatletter\ifx\SetFigFont\undefined%
\gdef\SetFigFont#1#2#3#4#5{%
  \reset@font\fontsize{#1}{#2pt}%
  \fontfamily{#3}\fontseries{#4}\fontshape{#5}%
  \selectfont}%
\fi\endgroup%
\resizebox{!}{22cm}{
\begin{picture}(0,0)%
\includegraphics{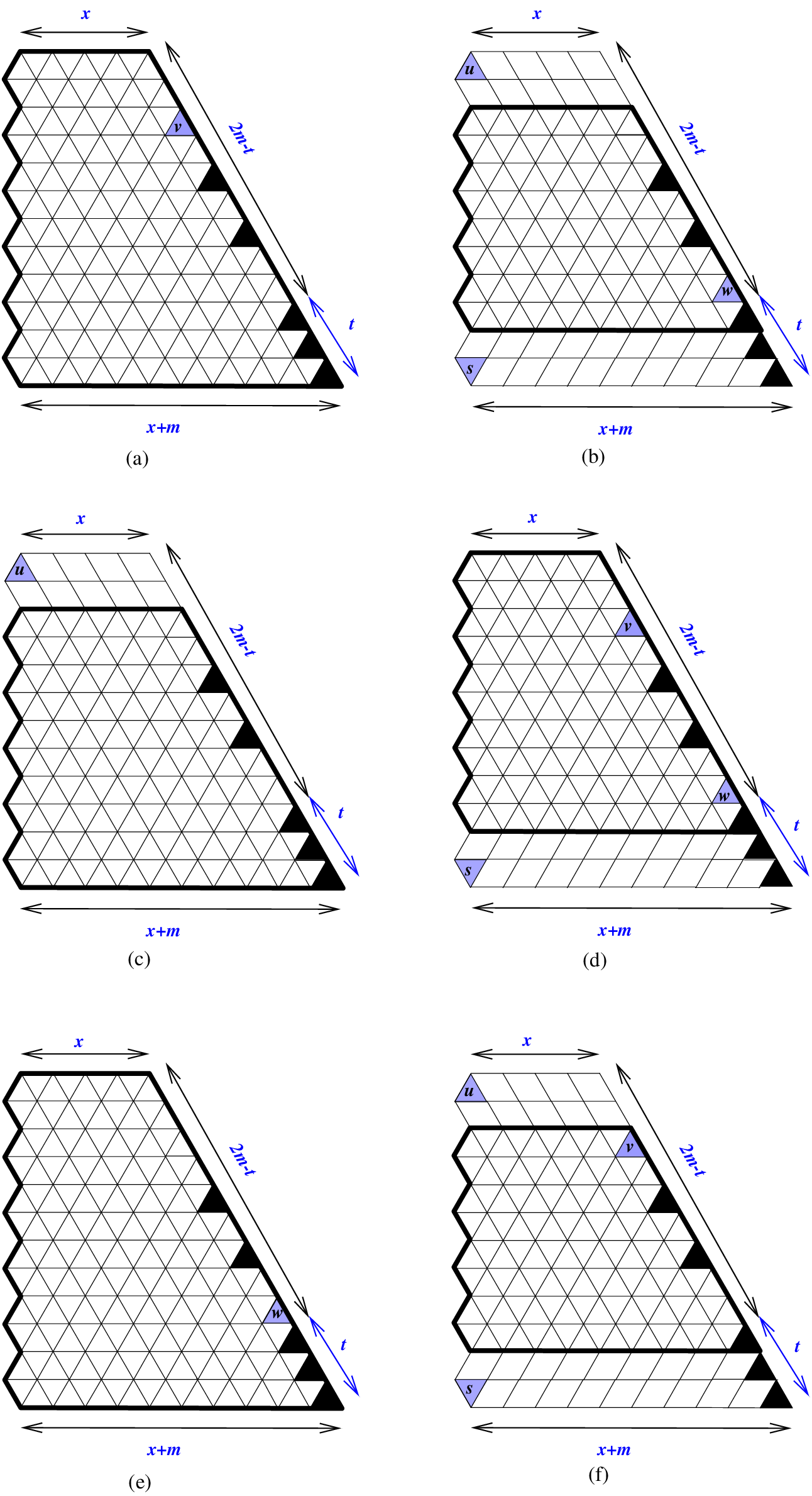}%
\end{picture}%
%
%

\begin{picture}(10323,19031)(964,-18476)
\put(9696,-8857){\makebox(0,0)[lb]{\smash{{\SetFigFont{14}{16.8}{\rmdefault}{\bfdefault}{\updefault}{\color[rgb]{1,1,1}$a_3$}%
}}}}
\put(9696,-15514){\makebox(0,0)[lb]{\smash{{\SetFigFont{14}{16.8}{\rmdefault}{\bfdefault}{\updefault}{\color[rgb]{1,1,1}$a_3$}%
}}}}
\put(9696,-2515){\makebox(0,0)[lb]{\smash{{\SetFigFont{14}{16.8}{\rmdefault}{\bfdefault}{\updefault}{\color[rgb]{1,1,1}$a_3$}%
}}}}
\put(3956,-8894){\makebox(0,0)[lb]{\smash{{\SetFigFont{14}{16.8}{\rmdefault}{\bfdefault}{\updefault}{\color[rgb]{1,1,1}$a_3$}%
}}}}
\put(3954,-2514){\makebox(0,0)[lb]{\smash{{\SetFigFont{14}{16.8}{\rmdefault}{\bfdefault}{\updefault}{\color[rgb]{1,1,1}$a_3$}%
}}}}
\put(5009,-4237){\makebox(0,0)[lb]{\smash{{\SetFigFont{14}{16.8}{\rmdefault}{\bfdefault}{\updefault}{\color[rgb]{1,1,1}$a_6$}%
}}}}
\put(4799,-3892){\makebox(0,0)[lb]{\smash{{\SetFigFont{14}{16.8}{\rmdefault}{\bfdefault}{\updefault}{\color[rgb]{1,1,1}$a_5$}%
}}}}
\put(4589,-3540){\makebox(0,0)[lb]{\smash{{\SetFigFont{14}{16.8}{\rmdefault}{\bfdefault}{\updefault}{\color[rgb]{1,1,1}$a_4$}%
}}}}
\put(10746,-4244){\makebox(0,0)[lb]{\smash{{\SetFigFont{14}{16.8}{\rmdefault}{\bfdefault}{\updefault}{\color[rgb]{1,1,1}$a_6$}%
}}}}
\put(10536,-3899){\makebox(0,0)[lb]{\smash{{\SetFigFont{14}{16.8}{\rmdefault}{\bfdefault}{\updefault}{\color[rgb]{1,1,1}$a_5$}%
}}}}
\put(10326,-3547){\makebox(0,0)[lb]{\smash{{\SetFigFont{14}{16.8}{\rmdefault}{\bfdefault}{\updefault}{\color[rgb]{1,1,1}$a_4$}%
}}}}
\put(5031,-10619){\makebox(0,0)[lb]{\smash{{\SetFigFont{14}{16.8}{\rmdefault}{\bfdefault}{\updefault}{\color[rgb]{1,1,1}$a_6$}%
}}}}
\put(4821,-10274){\makebox(0,0)[lb]{\smash{{\SetFigFont{14}{16.8}{\rmdefault}{\bfdefault}{\updefault}{\color[rgb]{1,1,1}$a_5$}%
}}}}
\put(4611,-9922){\makebox(0,0)[lb]{\smash{{\SetFigFont{14}{16.8}{\rmdefault}{\bfdefault}{\updefault}{\color[rgb]{1,1,1}$a_4$}%
}}}}
\put(10746,-10605){\makebox(0,0)[lb]{\smash{{\SetFigFont{14}{16.8}{\rmdefault}{\bfdefault}{\updefault}{\color[rgb]{1,1,1}$a_6$}%
}}}}
\put(3579,-1781){\makebox(0,0)[lb]{\smash{{\SetFigFont{14}{16.8}{\rmdefault}{\bfdefault}{\updefault}{\color[rgb]{1,1,1}$a_2$}%
}}}}
\put(10536,-10260){\makebox(0,0)[lb]{\smash{{\SetFigFont{14}{16.8}{\rmdefault}{\bfdefault}{\updefault}{\color[rgb]{1,1,1}$a_5$}%
}}}}
\put(9312,-1777){\makebox(0,0)[lb]{\smash{{\SetFigFont{14}{16.8}{\rmdefault}{\bfdefault}{\updefault}{\color[rgb]{1,1,1}$a_2$}%
}}}}
\put(10326,-9908){\makebox(0,0)[lb]{\smash{{\SetFigFont{14}{16.8}{\rmdefault}{\bfdefault}{\updefault}{\color[rgb]{1,1,1}$a_4$}%
}}}}
\put(3582,-8160){\makebox(0,0)[lb]{\smash{{\SetFigFont{14}{16.8}{\rmdefault}{\bfdefault}{\updefault}{\color[rgb]{1,1,1}$a_2$}%
}}}}
\put(10746,-17242){\makebox(0,0)[lb]{\smash{{\SetFigFont{14}{16.8}{\rmdefault}{\bfdefault}{\updefault}{\color[rgb]{1,1,1}$a_6$}%
}}}}
\put(9307,-8153){\makebox(0,0)[lb]{\smash{{\SetFigFont{14}{16.8}{\rmdefault}{\bfdefault}{\updefault}{\color[rgb]{1,1,1}$a_2$}%
}}}}
\put(10536,-16897){\makebox(0,0)[lb]{\smash{{\SetFigFont{14}{16.8}{\rmdefault}{\bfdefault}{\updefault}{\color[rgb]{1,1,1}$a_5$}%
}}}}
\put(3581,-14772){\makebox(0,0)[lb]{\smash{{\SetFigFont{14}{16.8}{\rmdefault}{\bfdefault}{\updefault}{\color[rgb]{1,1,1}$a_2$}%
}}}}
\put(10326,-16545){\makebox(0,0)[lb]{\smash{{\SetFigFont{14}{16.8}{\rmdefault}{\bfdefault}{\updefault}{\color[rgb]{1,1,1}$a_4$}%
}}}}
\put(9310,-14776){\makebox(0,0)[lb]{\smash{{\SetFigFont{14}{16.8}{\rmdefault}{\bfdefault}{\updefault}{\color[rgb]{1,1,1}$a_2$}%
}}}}
\put(5016,-17249){\makebox(0,0)[lb]{\smash{{\SetFigFont{14}{16.8}{\rmdefault}{\bfdefault}{\updefault}{\color[rgb]{1,1,1}$a_6$}%
}}}}
\put(4806,-16904){\makebox(0,0)[lb]{\smash{{\SetFigFont{14}{16.8}{\rmdefault}{\bfdefault}{\updefault}{\color[rgb]{1,1,1}$a_5$}%
}}}}
\put(4596,-16552){\makebox(0,0)[lb]{\smash{{\SetFigFont{14}{16.8}{\rmdefault}{\bfdefault}{\updefault}{\color[rgb]{1,1,1}$a_4$}%
}}}}
\put(3966,-15504){\makebox(0,0)[lb]{\smash{{\SetFigFont{14}{16.8}{\rmdefault}{\bfdefault}{\updefault}{\color[rgb]{1,1,1}$a_3$}%
}}}}
\end{picture}}
\caption{Obtaining a recurrence for the tiling generating functions of the quartered hexagons.}\label{Fig:Semitwodent9}
\end{figure}

We now can prove Theorem \ref{halfthm1}.

\begin{proof}[Proof of Theorem \ref{halfthm1}]
Let $t$ be the size of the maximal cluster of dents attaching to the lower-right corner of the region $Q_{x}=Q_{x}(\textbf{a})$, where $\textbf{a}=(a_i)_{i=1}^{m}$. The $t$-parameter varies from $0$ to $m$: $t=0$ if $a_m<2m$, and $t=m$ if $a_{i}=m+i$, for $i=1,2,\dots,m$. For example, the  quartered hexagon in Figure \ref{Fig:Semitwodent5}(a) has $t=3$.

We reformulate our identity (\ref{halfeq1}) as
\begin{align}\label{halfeq3}
\M(Q_{x}((a_i)_{i=1}^{m})&=q^{2(y-x)(\sum_{i=1}^{m}a_i- m^2)}\prod_{i=1}^{m}\frac{(q^{2(2y+a_i+1)};q^2)_{2i-a_i-1}}{(q^{2(2x+a_i+1)};q^2)_{2i-a_i-1}} \M(Q_{y}((a_i)_{i=1}^{m})
\end{align}
and denote
\begin{equation}
g_{x,y}((a_i)_{i=1}^{m})=q^{2(y-x)(\sum_{i=1}^{m}a_i- m^2)}\prod_{i=1}^{m}\frac{(q^{2(2y+a_i+1)};q^2)_{2i-a_i-1}}{(q^{2(2x+a_i+1)};q^2)_{2i-a_i-1}}.
\end{equation}
We plan to prove (\ref{halfeq3}) by induction on the statistic $p:=2m-t$. We note that $2m-t$ is always non-negative as $t\leq m$.

If $m=0$, then $Q_{x}(\emptyset)$ and $Q_{y}(\emptyset)$ become two degenerated regions. By convention, each has tiling generating function 1. Our identity simply becomes ``$1=1$."
If $m=1$, then  there are only two cases $\textbf{a}=(a_1)=(1)$ or $(2)$. If $a_1=1$, then both $Q_{x}((a_1))$ and $Q_{y}((a_1))$ has tiling generating function 1, and our identity is obviously true. If $a_1=2$, then  our region is exactly the halved hexagon $P_{x,1}$, and (\ref{halfeq3}) follows from Lemma \ref{proctorlem}.

If $t=m$, then  all of the dents are clustering to the lower-right corner of the region. Then our region, after removed forced lozenges, becomes a halved hexagon $P_{x,m}$ in Lemma \ref{proctorlem} (see Figure \ref{Fig:Semitwodent7}(e)). Again, (\ref{halfeq3}) follows from Lemma \ref{proctorlem}.

For the induction step, we assume that $m\geq 2$ and $t<m$, and that identity (\ref{halfeq3}) holds for any pair of quartered hexagons whose $p$-statistics are strictly less than $2m-t$.

First, we will show below that one could assume that  $a_1\geq 3$ and $t\geq 2$. Indeed, if $a_1=1$, then $a_2\geq 3$ by the tile-ability in Lemma \ref{tileability2}. Then we get forced lozenges along the first and second rows of unit triangles in $Q_x$ (see Figure \ref{Fig:Semitwodent7}(a)). After removing these forced lozenges (whose weights are all $1$), we get a `smaller'\footnote{Similar to the case of dented semi-hexagons, when we say a halved hexagon is``\emph{smaller}" than another halved hexagon if  its $p$-statistic is less than that of the latter one.} quartered hexagon with the same tiling generating function. We can do similarly to $Q_y=Q_{y}(\textbf{a})$, and (\ref{halfeq3})  follows from the induction hypothesis.

If $a_1=2$, then we can apply the Region-splitting Lemma (Lemma \ref{RS}) to split the region $Q_x$ into two smaller quartered hexagons as in Figure \ref{Fig:Semitwodent7}(c). (The cut is along the level 2 from the top of $Q_x$; the top portion is shaded.) Do similarly for $Q_y$, and (\ref{halfeq3}) follows from the Region-splitting Lemma and the induction hypothesis.

If $t=0$, then by the tile-ability in Lemma \ref{tileability2}, $a_m$ must be $2m-1$, then we have forced lozenges on the two bottom rows of $Q_x$ and $Q_y$ (illustrated in Figure  \ref{Fig:Semitwodent7}(b)). Removing these forced lozenges, we get a pair of smaller quartered hexagons. Again, (\ref{halfeq3}) follows  from  the induction hypothesis.

If $t=1$, then by definition of $t$, we have $s_{m}=2m$ and $s_{m-1}\leq2m-2$. By the Region-splitting Lemma, each of the regions $Q_x$ and $Q_y$ can be partitioned into two smaller quartered hexagons, as shown in Figure \ref{Fig:Semitwodent7}(d) (the lower portion is shaded). Then identity (\ref{halfeq3}) follows one more time from the induction hypothesis.

\medskip

In the rest of the proof, we are assuming besides that $a_1\geq 3$ and $t\geq 2$. We will use Kuo condensation in Lemma  \ref{kuothm1} to show that the expressions on both sides of (\ref{halfeq3}) satisfy the same recurrence. Then the theorem follows from  by the induction principle.

\medskip

We first work on the recurrence for the  tiling generating function of $Q_x$ on the left-hand side of (\ref{halfeq3}). We consider the dual graph $G$ of the region $R$ that is obtained from  $Q_x$ by filling the  first dent by an up-pointing unit triangle. The region $R$ now has one more up-pointing unit triangles than down-pointing triangles. We apply Kuo condensation in  Lemma \ref{kuothm1} to $G$ with the four vertices $u,v,w,s$  corresponding to the  shaded unit triangles  in Figure \ref{Fig:Semitwodent8}. In particular,  the $u$-triangle is the up-pointing triangle at the upper-left corner of $R$, the  $v$-triangle is the up-pointing triangle at the position $a_1$ (the previous position of the first dent in $Q_x$),  the $w$-triangle is at the last non-dent position on the right side of $R$, and the $s$-triangle is the down-pointing triangle at the lower-left corner of $R$. Let $\beta$ denote the the position of the $w$-triangle.

Figure \ref{Fig:Semitwodent9} tells us that the product of the tiling generating functions of the two regions in the top row is equal to the product of the tiling generating functions of the two regions in the middle row, plus the product of the tiling generating functions of the two regions in the bottom row. Working on the removal of forced lozenges (whose weights are all 1) as in shown the figure, we get the recurrence:
\begin{align}\label{halfeq4}
\M(Q_{x}(\textbf{a}))\M(Q_{x+1}(((a_{i})_{i=2}^{m-2}+\beta)^{**})&=\M(Q_{x+1}((a_{i})_{i=2}^{m})^{**})\M(Q_{x}((a_{i})_{i=1}^{m-2}+\beta))\notag\\
&+\M(Q_{x}((a_{i})_{i=2}^{m}+\beta))\M(Q_{x+1}(((a_{i})_{i=1}^{m-2})^{**}),
\end{align}
where we use the notation $\textbf{s}^{**}$ for the sequence obtained by subtracting $2$ from each term of the sequence $\textbf{s}$ (we still use the notation $\textbf{s}+\beta$ for the sequence obtained from including the term $\beta$ to $\textbf{s}$ and rearranging in increasing order).
It is easy to see that if $Q_{x}(\textbf{a})$ is tile-able, then the other five regions in the recurrence are also tile-able by Lemma \ref{tileability2}. Moreover, these five regions are all strictly smaller than $Q_x=Q_{x}(\textbf{a})$.

To finish the proof, we need to show that the expression on the right-hand side of (\ref{halfeq3}), i.e.,
\[g_{x,y}(\textbf{a})\M(Q_{y}((a_i)_{i=1}^{m}))=q^{2(y-x)(\sum_{i=1}^{m}a_i- m^2)}\prod_{i=1}^{m}\frac{(q^{2(2y+a_i+1)};q^2)_{2i-a_i-1}}{(q^{2(2x+a_i+1)};q^2)_{2i-a_i-1}} \M(Q_{y}((a_i)_{i=1}^{m}),\]
 also satisfies recurrence (\ref{halfeq4}) above. Equivalently,  we need to verify that

\begin{align}\label{halfeq5}
A\cdot \M(Q_{y}(\textbf{a}))\M(Q_{y+1}(((a_{i})_{i=2}^{m-2}+\beta)^{**})&=B\cdot \M(Q_{y+1}((a_{i})_{i=2}^{m})^{**})\M(Q_{y}((a_{i})_{i=1}^{m-2}+\beta))\notag\\
&+C\cdot \M(Q_{y}((a_{i})_{i=2}^{m}+\beta))\M(Q_{y+1}(((a_{i})_{i=1}^{m-2})^{**}),
\end{align}
where
\begin{align}
A=g_{x,y}(\textbf{a})\cdot g_{x+1,y+1}(((a_{i})_{i=2}^{m-2}+\beta)^{**}),
\end{align}
\begin{align}
B=g_{x+1,y+1}((a_{i})_{i=2}^{m})^{**}) \cdot g_{x,y}(((a_{i})_{i=1}^{m-2}+\beta),
\end{align}
\begin{align}
C=g_{x,y}((a_{i})_{i=2}^{m}+\beta) \cdot g_{x+1,y+1}(((a_{i})_{i=1}^{m-2})^{**}).
\end{align}

By definition, one could routinely verify that $A=B=C$. Then (\ref{halfeq5}) reduces to
\begin{align}\label{halfeq6}
 \M(Q_{y}(\textbf{a}))\M(Q_{y+1}(((a_{i})_{i=2}^{m-2}+\beta)^{**})&= \M(Q_{y+1}((a_{i})_{i=2}^{m})^{**})\M(Q_{y}((a_{i})_{i=1}^{m-2}+\beta))\notag\\
&+\M(Q_{y}((a_{i})_{i=1}^{m}+\beta))\M(Q_{y+1}(((a_{i})_{i=1}^{m-2})^{**}).
\end{align}
However, this recurrence follows immediately from recurrence (\ref{halfeq4}) by replacing $x$ by $y$. This finishes our proof.
\end{proof}

Theorem \ref{halfthm2} can be proved in the same manner as Theorem \ref{halfthm1}, using Lemma \ref{proctorlem2}. Even though the lozenges are weighted differently in Theorems \ref{halfthm1} and \ref{halfthm2}, the Kuo condensation works essentially the same as the forced lozenges all have weight 1. We omit the proof of Theorem \ref{halfthm2} here.




\section{Appendix: Proofs of Lemmas \ref{semilem1}--\ref{proctorlem2}}\label{Sec:Appendix}

We first show briefly here the proof of Lemma \ref{semilem1}. The proof of Lemma \ref{semilem2} is essentially similar and will be left as an exercise for the reader.

\begin{figure}\centering
\includegraphics[width=13cm]{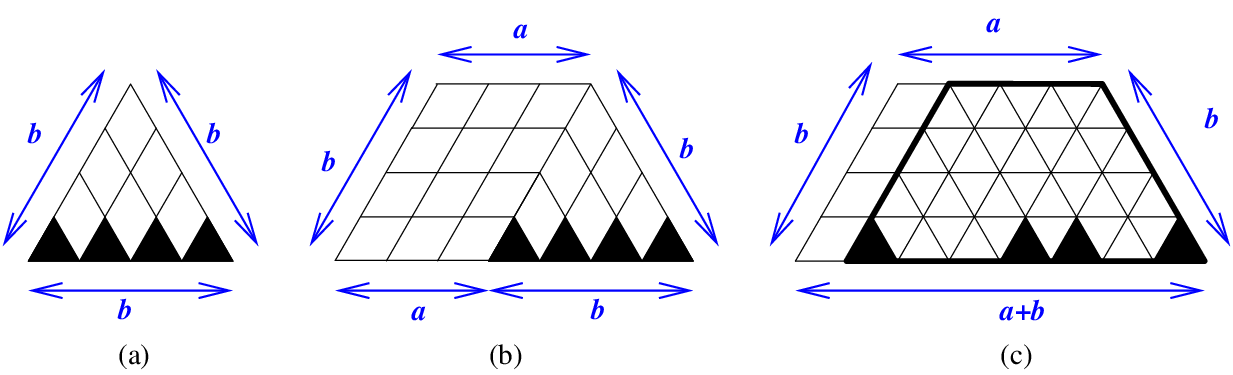}
\caption{Several special cases of the semi-hexagon with dents on the base.}\label{Fig:base2}
\end{figure}

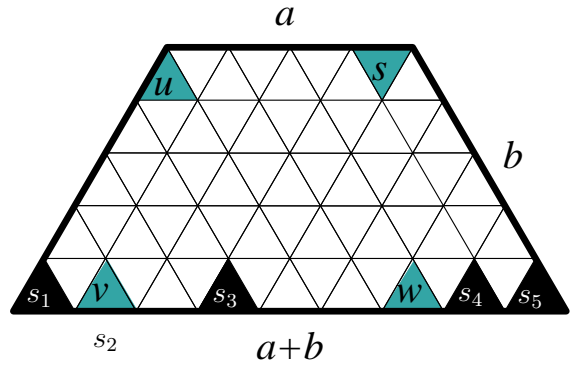
\begin{figure}\centering
\setlength{\unitlength}{3947sp}%
\begingroup\makeatletter\ifx\SetFigFont\undefined%
\gdef\SetFigFont#1#2#3#4#5{%
  \reset@font\fontsize{#1}{#2pt}%
  \fontfamily{#3}\fontseries{#4}\fontshape{#5}%
  \selectfont}%
\fi\endgroup%
\resizebox{!}{5cm}{
\begin{picture}(0,0)%
\includegraphics{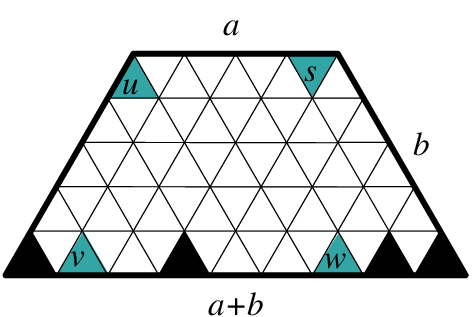}%
\end{picture}%
%
%

\begin{picture}(3770,2532)(5474,-1731)
\put(8875,-1343){\makebox(0,0)[lb]{\smash{{\SetFigFont{12}{14.4}{\rmdefault}{\mddefault}{\itdefault}{\color[rgb]{1,1,1}$s_5$}%
}}}}
\put(8478,-1328){\makebox(0,0)[lb]{\smash{{\SetFigFont{12}{14.4}{\rmdefault}{\mddefault}{\itdefault}{\color[rgb]{1,1,1}$s_4$}%
}}}}
\put(6843,-1336){\makebox(0,0)[lb]{\smash{{\SetFigFont{12}{14.4}{\rmdefault}{\mddefault}{\itdefault}{\color[rgb]{1,1,1}$s_3$}%
}}}}
\put(5605,-1336){\makebox(0,0)[lb]{\smash{{\SetFigFont{12}{14.4}{\rmdefault}{\mddefault}{\itdefault}{\color[rgb]{1,1,1}$s_1$}%
}}}}
\put(6046,-1636){\makebox(0,0)[lb]{\smash{{\SetFigFont{12}{14.4}{\rmdefault}{\mddefault}{\itdefault}{\color[rgb]{0,0,0}$s_2$}%
}}}}
\end{picture}}
\caption{How to apply Kuo condensation to the semi-hexagon with dents on the base.}\label{Fig:base4}
\end{figure}

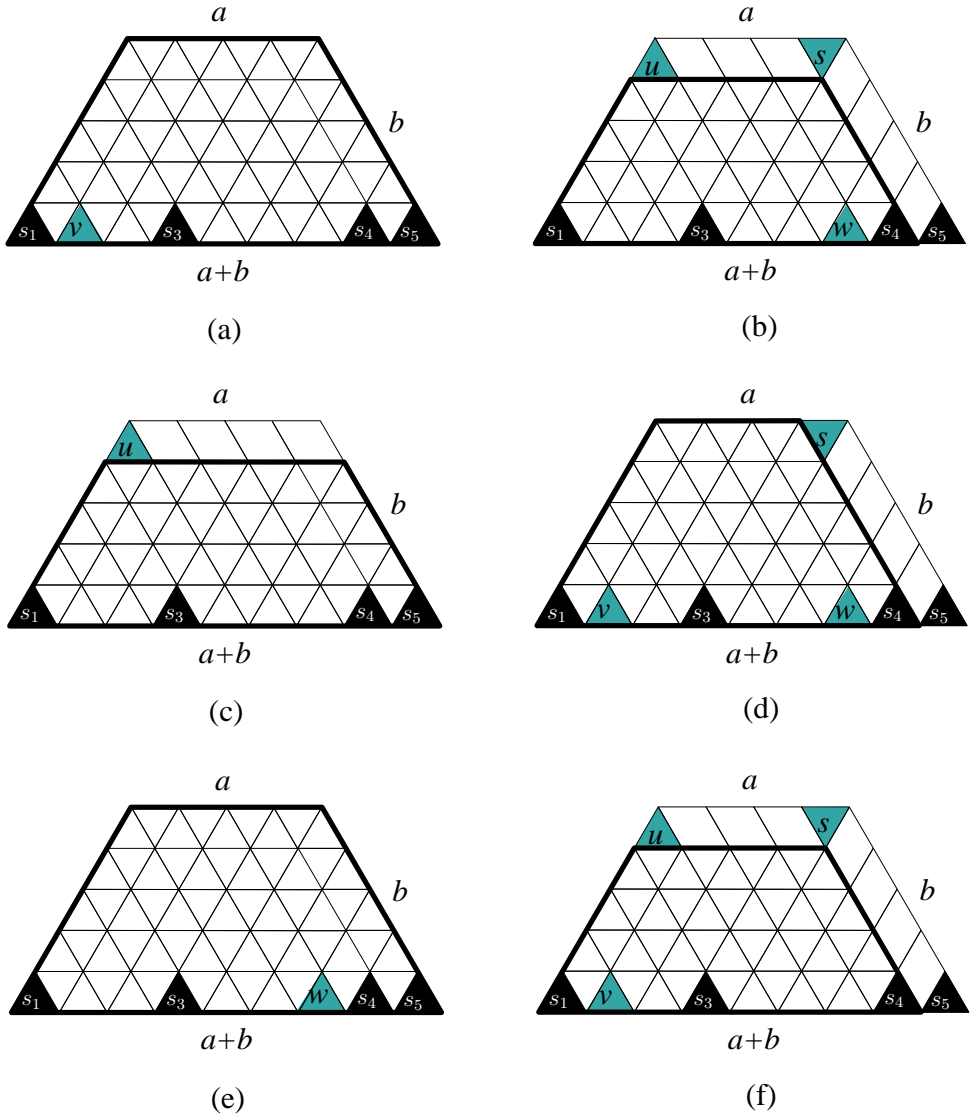
\begin{figure}\centering
\setlength{\unitlength}{3947sp}%
\begingroup\makeatletter\ifx\SetFigFont\undefined%
\gdef\SetFigFont#1#2#3#4#5{%
  \reset@font\fontsize{#1}{#2pt}%
  \fontfamily{#3}\fontseries{#4}\fontshape{#5}%
  \selectfont}%
\fi\endgroup%
\resizebox{!}{15cm}{
\begin{picture}(0,0)%
\includegraphics{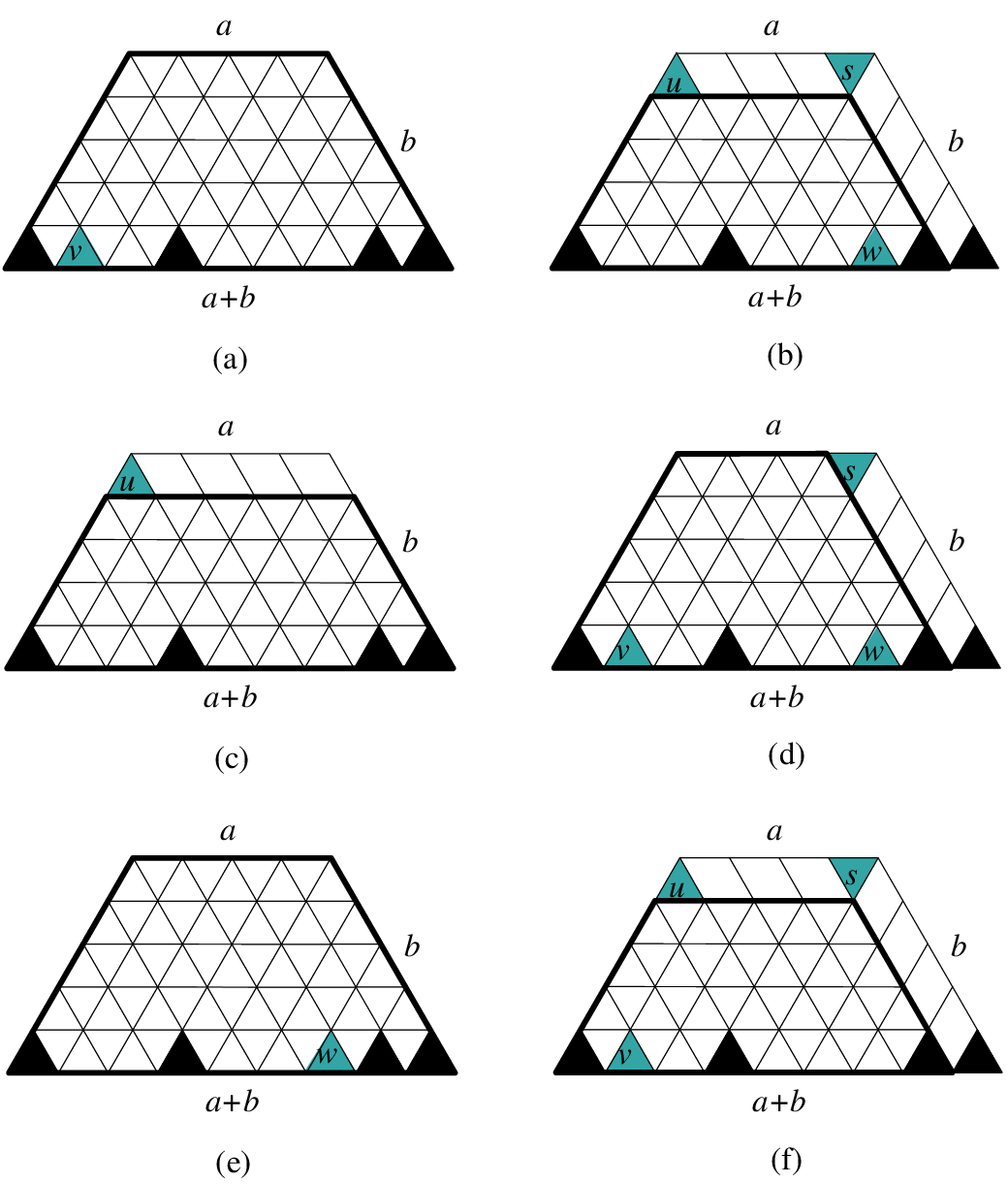}%
\end{picture}%
%
%

\begin{picture}(8307,9766)(959,-8953)
\put(8890,-4642){\makebox(0,0)[lb]{\smash{{\SetFigFont{12}{14.4}{\rmdefault}{\mddefault}{\itdefault}{\color[rgb]{1,1,1}$s_5$}%
}}}}
\put(8493,-4627){\makebox(0,0)[lb]{\smash{{\SetFigFont{12}{14.4}{\rmdefault}{\mddefault}{\itdefault}{\color[rgb]{1,1,1}$s_4$}%
}}}}
\put(6858,-4635){\makebox(0,0)[lb]{\smash{{\SetFigFont{12}{14.4}{\rmdefault}{\mddefault}{\itdefault}{\color[rgb]{1,1,1}$s_3$}%
}}}}
\put(5620,-4635){\makebox(0,0)[lb]{\smash{{\SetFigFont{12}{14.4}{\rmdefault}{\mddefault}{\itdefault}{\color[rgb]{1,1,1}$s_1$}%
}}}}
\put(6873,-7972){\makebox(0,0)[lb]{\smash{{\SetFigFont{12}{14.4}{\rmdefault}{\mddefault}{\itdefault}{\color[rgb]{1,1,1}$s_3$}%
}}}}
\put(8508,-7964){\makebox(0,0)[lb]{\smash{{\SetFigFont{12}{14.4}{\rmdefault}{\mddefault}{\itdefault}{\color[rgb]{1,1,1}$s_4$}%
}}}}
\put(8905,-7979){\makebox(0,0)[lb]{\smash{{\SetFigFont{12}{14.4}{\rmdefault}{\mddefault}{\itdefault}{\color[rgb]{1,1,1}$s_5$}%
}}}}
\put(1090,-1339){\makebox(0,0)[lb]{\smash{{\SetFigFont{12}{14.4}{\rmdefault}{\mddefault}{\itdefault}{\color[rgb]{1,1,1}$s_1$}%
}}}}
\put(8875,-1343){\makebox(0,0)[lb]{\smash{{\SetFigFont{12}{14.4}{\rmdefault}{\mddefault}{\itdefault}{\color[rgb]{1,1,1}$s_5$}%
}}}}
\put(1120,-7975){\makebox(0,0)[lb]{\smash{{\SetFigFont{12}{14.4}{\rmdefault}{\mddefault}{\itdefault}{\color[rgb]{1,1,1}$s_1$}%
}}}}
\put(2358,-7975){\makebox(0,0)[lb]{\smash{{\SetFigFont{12}{14.4}{\rmdefault}{\mddefault}{\itdefault}{\color[rgb]{1,1,1}$s_3$}%
}}}}
\put(3993,-7967){\makebox(0,0)[lb]{\smash{{\SetFigFont{12}{14.4}{\rmdefault}{\mddefault}{\itdefault}{\color[rgb]{1,1,1}$s_4$}%
}}}}
\put(4390,-7982){\makebox(0,0)[lb]{\smash{{\SetFigFont{12}{14.4}{\rmdefault}{\mddefault}{\itdefault}{\color[rgb]{1,1,1}$s_5$}%
}}}}
\put(8478,-1328){\makebox(0,0)[lb]{\smash{{\SetFigFont{12}{14.4}{\rmdefault}{\mddefault}{\itdefault}{\color[rgb]{1,1,1}$s_4$}%
}}}}
\put(6843,-1336){\makebox(0,0)[lb]{\smash{{\SetFigFont{12}{14.4}{\rmdefault}{\mddefault}{\itdefault}{\color[rgb]{1,1,1}$s_3$}%
}}}}
\put(5605,-1336){\makebox(0,0)[lb]{\smash{{\SetFigFont{12}{14.4}{\rmdefault}{\mddefault}{\itdefault}{\color[rgb]{1,1,1}$s_1$}%
}}}}
\put(1105,-4638){\makebox(0,0)[lb]{\smash{{\SetFigFont{12}{14.4}{\rmdefault}{\mddefault}{\itdefault}{\color[rgb]{1,1,1}$s_1$}%
}}}}
\put(2343,-4638){\makebox(0,0)[lb]{\smash{{\SetFigFont{12}{14.4}{\rmdefault}{\mddefault}{\itdefault}{\color[rgb]{1,1,1}$s_3$}%
}}}}
\put(3978,-4630){\makebox(0,0)[lb]{\smash{{\SetFigFont{12}{14.4}{\rmdefault}{\mddefault}{\itdefault}{\color[rgb]{1,1,1}$s_4$}%
}}}}
\put(4375,-4645){\makebox(0,0)[lb]{\smash{{\SetFigFont{12}{14.4}{\rmdefault}{\mddefault}{\itdefault}{\color[rgb]{1,1,1}$s_5$}%
}}}}
\put(4360,-1346){\makebox(0,0)[lb]{\smash{{\SetFigFont{12}{14.4}{\rmdefault}{\mddefault}{\itdefault}{\color[rgb]{1,1,1}$s_5$}%
}}}}
\put(3963,-1331){\makebox(0,0)[lb]{\smash{{\SetFigFont{12}{14.4}{\rmdefault}{\mddefault}{\itdefault}{\color[rgb]{1,1,1}$s_4$}%
}}}}
\put(2328,-1339){\makebox(0,0)[lb]{\smash{{\SetFigFont{12}{14.4}{\rmdefault}{\mddefault}{\itdefault}{\color[rgb]{1,1,1}$s_3$}%
}}}}
\put(5635,-7972){\makebox(0,0)[lb]{\smash{{\SetFigFont{12}{14.4}{\rmdefault}{\mddefault}{\itdefault}{\color[rgb]{1,1,1}$s_1$}%
}}}}
\end{picture}}
\caption{Obtaining recurrence for tiling generating functions  of the semi-hexagons with dents on the base.}\label{Fig:base3}
\end{figure}

\begin{proof}[Proof of Lemma \ref{semilem1}]
We define $t=b-l$, where $l$ is the size of the maximal dent cluster attaching to the lower-right corner of the semi-hexagon $S=S_{a,b}(\textbf{a})$. We prove the lemma by induction on $a+b+t$. 

The base cases are the situations when at least one of the parameters $a,b,t$ equal to $0$. If $b=0$, then our region is degenerated, and the tiling formula is obviously true. If $a=0$ or $t=0$, then our region has only one tiling, as shown in Figure \ref{Fig:base2}(a) or (b), respectively. It is easy to verify the tiling formula in these cases.

 For the induction step we assume that $a,b,t>0$ and that the lemma holds for any semi-hexagons whose sum of $a$-, $b$-, and $t$-parameters is strictly less than $a+b+t$. It is easy to see that we can assume $s_1=1$ and $s_m=x+m$. Otherwise, one can remove forced lozenges from $S$ to obtain a smaller region of the same type (see Figure \ref{Fig:base2}(c) for the case $s_1>1$; the case $s_b<a+b$ is similar by symmetry), and the lemma follows from the induction hypothesis.
 
 Assume that $s_k$ is the first dent position so that there is no dent on right of its (in ptarticular, $s_{k+1}>s_k+1$). We consider the region $R$ obtained from $S$ by  filling the $s_k$-dent. In particular, $R$ has one more up-pointing triangles than down-pointing triangles. We apply  Kuo condensation in Lemma \ref{kuothm1} to the dual graph $G$ of $R$ with the four vertices $u,v,w,s$ as shown in Figure \ref{Fig:base4}. More precisely, the $u$-triangle is the up-pointing unit triangles at the upper-left corner and the $s$-triangle is the down-pointing unit triangle at the upper-right corner. The $v$-triangle is at the position of $\alpha=s_k$ and the $w$ position is at the position $\beta=s_{b-t+1}-1$. Considering the removal of forced lozenges as in Figure  \ref{Fig:base3}, we get the following recurrence
\begin{align}
\M(S_{a,b}(\textbf{s}))\M(S_{a,b-1}((s_i)_{i=1}^{b-1}-\alpha+\beta))&=\M(S_{a+1,b-1}(\textbf{s}-\alpha))\M(S_{a-1,b}((s_i)_{i=1}^{b-1}+\beta))\notag\\
&+\M(S_{a,b}(\textbf{s}-\alpha+\beta))\M(S_{a,b-1}((s_i)_{i=1}^{b-1})).
\end{align}
(The weights of forced lozenges cancel out.) Then the lemma follows from the induction principle.

\end{proof}

Next, we show the proof of Lemma \ref{proctorlem} (Lemma \ref{proctorlem2} could be proved in the same manner).

\begin{figure}\centering
\includegraphics[width=9.5cm]{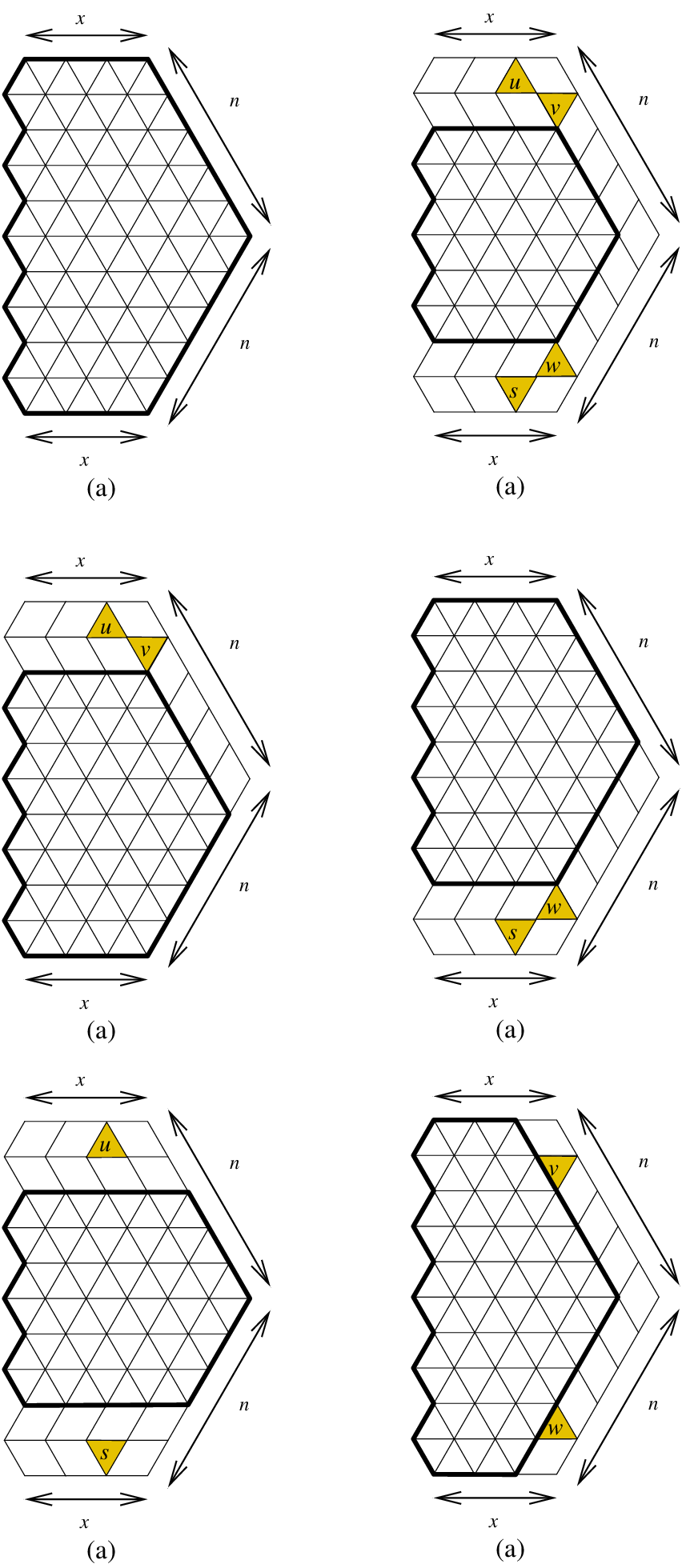}
\caption{Obtaining recurrence for tiling generating functions of halved hexagons.}\label{Fig:halved2}
\end{figure}

\begin{proof}[Proof of Lemma \ref{proctorlem}]
We prove by induction on $x+n$.  The base cases are the situations $x=0$ and $n\leq 1$.

When $x=0$, the region $P_{x,n}$ has only one tiling consisting of vertical lozenges; when $n=0$, then the region is degenerated. It is easy to verify our identity in these cases. If $n=1$, then our region become a  hexagon of side-lengths $x,1,1,x,1,1$. It is easy to see that the hexagon has exactly $x$ tilings, each consists of one  vertical   lozenge, $x$ left lozenges, and $x$ right lozenges. One could calculate the tilling generating function and then easily verify the identity in this case.

For the induction step, we assume that $x>0$ and $n>1$ and that the tiling formula holds for any  halved hexagons whose sum of $x$- and $n$-parameters is strictly less than $x+n$. Applying Kuo condensation in Lemma \ref{kuothm0} to the dual graph $G$ of the halved hexagon $P_{x,n}$, as shown in Figure \ref{Fig:halved2}. We get the following recurrence:
\begin{align}
\M(P_{x,n})\M(P_{x,n-2})=\left(\frac{q^{2x+n}+q^{-2x-n}}{2}\right)\M(P_{x,n-1})\M(P_{x,n-1})+\M(P_{x+1,n-2})\M(P_{x-1,n}).
\end{align}
We note that the factor $\frac{q^{2x+n}+q^{-2x-n}}{2}$ comes from the weight of the right most  vertical lozenge; the weights of all other forced lozenges cancel out. Then  the lemma follows from the induction principle.
\end{proof}

\bibliographystyle{plain}
\bibliography{Semitwodent}

\end{document}